\documentclass[11pt]{article}

\usepackage{enumerate,xspace}
\usepackage{amsmath,amssymb,wasysym}
\usepackage[all]{xy}
\usepackage{proof}
\usepackage[svgnames]{xcolor}
\usepackage{pict2e} 
\usepackage{tikz}
\usepackage{stmaryrd}
\usepackage{mathtools}
\usepackage{latexsym}
\usepackage{dsfont}
\usepackage{multicol}
\usepackage{cmll}
\usepackage{multirow}
\usepackage{longtable}
\usepackage[sort,nocompress]{cite}

\usepackage{todonotes}

\usepackage{lscape}
\usepackage{array}

\usepackage{hyperref} 
\hypersetup{
    colorlinks,
    citecolor=red,
    filecolor=red,
    linkcolor=blue,
    urlcolor=red
}

\newtheorem{observation}{Remark}[section]
\newtheorem{lemma}[observation]{Lemma}  
\newtheorem{theorem}[observation]{Theorem}
\newtheorem{definition}[observation]{Definition}
\newtheorem{example}[observation]{Example}
\newtheorem{remark}[observation]{Remark}

\newtheorem{proposition}[observation]{Proposition} 
\newtheorem{corollary}[observation]{Corollary}

\newcommand{\wand}{\ensuremath{
  \mathrel{\vbox{\offinterlineskip\ialign{
    \hfil##\hfil\cr
    $\star$\cr
    \noalign{\kern-1ex}
    $\vert$\cr
}}}}}

\makeatletter

\newdimen\w@dth

\def\setw@dth#1#2{\setbox\z@\hbox{\scriptsize $#1$}\w@dth=\wd\z@
\setbox\@ne\hbox{\scriptsize $#2$}\ifnum\w@dth<\wd\@ne \w@dth=\wd\@ne \fi
\advance\w@dth by 1.2em}

\def\t@^#1_#2{\allowbreak\def\n@one{#1}\def\n@two{#2}\mathrel
{\setw@dth{#1}{#2}
\mathop{\hbox to \w@dth{\rightarrowfill}}\limits
\ifx\n@one\empty\else ^{\box\z@}\fi
\ifx\n@two\empty\else _{\box\@ne}\fi}}
\def\t@@^#1{\@ifnextchar_ {\t@^{#1}}{\t@^{#1}_{}}}

\def\t@left^#1_#2{\def\n@one{#1}\def\n@two{#2}\mathrel{\setw@dth{#1}{#2}
\mathop{\hbox to \w@dth{\leftarrowfill}}\limits
\ifx\n@one\empty\else ^{\box\z@}\fi
\ifx\n@two\empty\else _{\box\@ne}\fi}}
\def\t@@left^#1{\@ifnextchar_ {\t@left^{#1}}{\t@left^{#1}_{}}}

\def\two@^#1_#2{\def\n@one{#1}\def\n@two{#2}\mathrel{\setw@dth{#1}{#2}
\mathop{\vcenter{\hbox to \w@dth{\rightarrowfill}\kern-1.7ex
                 \hbox to \w@dth{\rightarrowfill}}%
       }\limits
\ifx\n@one\empty\else ^{\box\z@}\fi
\ifx\n@two\empty\else _{\box\@ne}\fi}}
\def\tw@@^#1{\@ifnextchar_ {\two@^{#1}}{\two@^{#1}_{}}}

\def\tofr@^#1_#2{\def\n@one{#1}\def\n@two{#2}\mathrel{\setw@dth{#1}{#2}
\mathop{\vcenter{\hbox to \w@dth{\rightarrowfill}\kern-1.7ex
                 \hbox to \w@dth{\leftarrowfill}}%
       }\limits
\ifx\n@one\empty\else ^{\box\z@}\fi
\ifx\n@two\empty\else _{\box\@ne}\fi}}
\def\t@fr@^#1{\@ifnextchar_ {\tofr@^{#1}}{\tofr@^{#1}_{}}}

\newdimen\W@dth
\def\setW@dth#1#2{\setbox\z@\hbox{$#1$}\W@dth=\wd\z@
\setbox\@ne\hbox{$#2$}\ifnum\W@dth<\wd\@ne \W@dth=\wd\@ne \fi
\advance\W@dth by 1.2em}

\def\T@^#1_#2{\allowbreak\def\N@one{#1}\def\N@two{#2}\mathrel
{\setW@dth{#1}{#2}
\mathop{\hbox to \W@dth{\rightarrowfill}}\limits
\ifx\N@one\empty\else ^{\box\z@}\fi
\ifx\N@two\empty\else _{\box\@ne}\fi}}
\def\T@@^#1{\@ifnextchar_ {\T@^{#1}}{\T@^{#1}_{}}}

\def\T@left^#1_#2{\def\N@one{#1}\def\N@two{#2}\mathrel{\setW@dth{#1}{#2}
\mathop{\hbox to \W@dth{\leftarrowfill}}\limits
\ifx\N@one\empty\else ^{\box\z@}\fi
\ifx\N@two\empty\else _{\box\@ne}\fi}}
\def\T@@left^#1{\@ifnextchar_ {\T@left^{#1}}{\T@left^{#1}_{}}}

\def\Tofr@^#1_#2{\def\N@one{#1}\def\N@two{#2}\mathrel{\setW@dth{#1}{#2}
\mathop{\vcenter{\hbox to \W@dth{\rightarrowfill}\kern-1.7ex
                 \hbox to \W@dth{\leftarrowfill}}%
       }\limits
\ifx\N@one\empty\else ^{\box\z@}\fi
\ifx\N@two\empty\else _{\box\@ne}\fi}}
\def\T@fr@^#1{\@ifnextchar_ {\Tofr@^{#1}}{\Tofr@^{#1}_{}}}

\def\Two@^#1_#2{\def\N@one{#1}\def\N@two{#2}\mathrel{\setW@dth{#1}{#2}
\mathop{\vcenter{\hbox to \W@dth{\rightarrowfill}\kern-1.7ex
                 \hbox to \W@dth{\rightarrowfill}}%
       }\limits
\ifx\N@one\empty\else ^{\box\z@}\fi
\ifx\N@two\empty\else _{\box\@ne}\fi}}
\def\Tw@@^#1{\@ifnextchar_ {\Two@^{#1}}{\Two@^{#1}_{}}}

\def\to{\@ifnextchar^ {\t@@}{\t@@^{}}}
\def\from{\@ifnextchar^ {\t@@left}{\t@@left^{}}}
\def\tofro{\@ifnextchar^ {\t@fr@}{\t@fr@^{}}}
\def\To{\@ifnextchar^ {\T@@}{\T@@^{}}}
\def\From{\@ifnextchar^ {\T@@left}{\T@@left^{}}}
\def\Two{\@ifnextchar^ {\Tw@@}{\Tw@@^{}}}
\def\Tofro{\@ifnextchar^ {\T@fr@}{\T@fr@^{}}}

\makeatother

\title{Derivations as Algebras}
\author{Jean-Simon Pacaud Lemay and Chiara Sava}
\date{}

\begin{document}
\allowdisplaybreaks
\maketitle

\begin{abstract} Differential categories provide the categorical foundations for the algebraic approaches to differentiation. They have been successful in formalizing various important concepts related to differentiation, such as, in particular, derivations. In this paper, we show that the differential modality of a differential category lifts to a monad on the arrow category and, moreover, that the algebras of this monad are precisely derivations. Furthermore, in the presence of finite biproducts, the differential modality in fact lifts to a differential modality on the arrow category. In other words, the arrow category of a differential category is again a differential category. As a consequence, derivations also form a tangent category, and derivations on free algebras form a cartesian differential category. 
\end{abstract}

 \noindent \small \textbf{Acknowledgements.} For this research, Lemay was funded by an ARC DECRA award (\# DE230100303) and this material is based upon work supported by the AFOSR under award number FA9550-24-1-0008; while Sava was supported by the Charles University Grant Agency (GA UK) project n° 222923 and by the Primus grant PRIMUS/23/SCI/006. The authors would also like to thank the Australian Category Theory group at Macquarie University who helped fund Sava with the support of a Scott Russell Johnson Postgraduate Scholarship. 

 \tableofcontents

 \newpage 

\section{Introduction}

Differential categories, introduced by Blute, Cockett, and Seely in \cite{blute2006differential}, provide a powerful categorical framework for the algebraic foundations of differentiation, as well as for the categorical semantics of differential linear logic \cite{ehrhard2017introduction}. At their core, differential categories isolate and axiomatize the essential structural properties of differentiation in a way that is both conceptually clean and mathematically flexible. Concretely, a differential category (Def \ref{def:diffcat}) is a symmetric monoidal category equipped with a differential modality: a monad $\mathsf{S}$ together with a deriving transformation $\mathsf{d}$, whose axioms categorify fundamental identities of differential calculus, such as the chain rule and the Leibniz rule. This structure admits an intuitive insight: $\mathsf{S}(A)$ may be viewed as an algebra of differentiable functions with input in $A$, while $\mathsf{d}$ plays the role of an abstract differential operator sending functions to their derivatives. Differential categories encompass a rich collection of meaningful and well-studied examples which capture important models of differentiation. These include polynomial differentiation \cite{blute2006differential}, where $\mathsf{S}=\mathsf{Sym}$ is the free symmetric algebra monad (Ex \ref{ex:sym}); smooth functions \cite{cruttwell2019integral}, where $\mathsf{S}=\mathsf{S}^\infty$ is the free $\mathcal{C}^\infty$-ring monad (Ex \ref{ex:s-inf}); as well as more exotic examples from computer science, such as differentiation over finiteness spaces or Köthe spaces \cite{ehrhard2017introduction}.

\smallskip

The theory of differential categories has been quite successful in formalizing various key differentiation related concept such as, in particular, \textit{derivations}. In classical algebra, a derivation is the generalization of the differential operator which, recall, is a linear operator from an algebra to a module which satisfies the Leibniz rule. Derivations are a  fundamental and ubiquitous concept, with deep applications across algebra, differential geometry, algebraic geometry, and beyond. 

\smallskip

In \cite{blute2015derivations}, inspired by earlier work on Kähler differentials \cite{blute2011kahler}, Blute, Lucyshyn-Wright, and O’Neill introduced a categorical generalization of derivations within a differential category, defining the notion of an $\mathsf{S}$-derivation relative to a differential modality $\mathsf{S}$. An $\mathsf{S}$-derivation (Def \ref{def:S-der}) is a map from an $\mathsf{S}$-algebra (in the monad sense) to a module, axiomatized not by the Leibniz rule, but instead by the chain rule. Remarkably, every $\mathsf{S}$-derivation automatically satisfies the Leibniz rule as a consequence, and hence recovers the classical notion of derivation. However, this is much more than just a reformulation: $\mathsf{S}$-derivations are precisely those derivations that are compatible with the class of differentiable functions encoded by the differential modality $\mathsf{S}$. For example, $\mathsf{Sym}$-derivation correspond to ordinary derivations, since they satisfy a chain rule with polynomials (Ex \ref{ex:der-sym}), while $\mathsf{S}^\infty$-derivations correspond to $\mathcal{C}^\infty$-derivations \cite{dubuc19841,joyce2011introduction}, which satisfy a chain rule with real smooth functions (Ex \ref{ex:der-smooth}). 

\smallskip

A substantial body of classical results and concepts about derivations admits meaningful generalizations to $\mathsf{S}$-derivations. For example, derivations can be characterized as algebra morphisms into a suitable semi-direct product of an algebra and a module; an analogous statement holds for $\mathsf{S}$-derivations \cite[Prop 5.21]{blute2015derivations}. Likewise, the classical construction of the universal derivation via Kähler differentials extends, under mild colimit assumptions, to yield universal $\mathsf{S}$-derivations for $\mathsf{S}$-algebras \cite[Thm 5.23]{blute2015derivations}. This in turn enables the development of de Rham cohomology internal to differential categories \cite{o2017smoothness}. In another direction, $\mathsf{S}$-algebras equipped with $\mathsf{S}$-derivations to themselves generalize differential algebras in a differential category \cite{lemay2019differential}. 

\smallskip

The main contribution of this paper is a new and conceptually striking characterization of $\mathsf{S}$-derivations: we show that they are precisely the algebras of a monad on the arrow category. More specifically, we first show that any differential modality $\mathsf{S}$ lifts canonically to a monad $\overline{\mathsf{S}}$ on the arrow category (Prop \ref{prop:monad}). The functor $\overline{\mathsf{S}}$ is constructed directly from the deriving transformation $\mathsf{d}$, and its monad structure is well defined by the chain rule axiom. We then prove that the algebras of this lifted monad $\overline{\mathsf{S}}$ are exactly $\mathsf{S}$-derivations.

\smallskip

This characterization is useful for several reasons. It holds in any differential category, without requiring additional structure such as biproducts or colimits, in contrast to other approaches \cite{blute2015derivations}. Moreover, it offers a genuinely new perspective on the well established notion of derivations, which is only visible through the lens of differential categories.  When biproducts are present, this perspective becomes even richer. We show that the arrow category admits a special monoidal structure (Lemma \ref{lem:arrow-monoidal}) which makes $\overline{\mathsf{S}}$ into a differential modality (Thm \ref{thm:arrow-diff-cat}), which is well-defined this time thanks to the Leibniz rule. As a consequence, the arrow category of a differential category with biproducts is itself a differential category, providing a new and systematic source of examples. In this setting, commutative monoids in the arrow category correspond exactly to classical derivations (Prop \ref{prop:mon=der-arrow}), reinforcing the tight connection with our monadic characterization.

\smallskip

These results have meaningful implications for other related categorical frameworks of differentiation. Indeed, differential categories are closely linked to cartesian differential categories \cite{blute2009cartesian}, which provide the categorical foundations for multivariable differential calculus, and tangent categories \cite{cockett2014differential}, which provide the categorical foundations for differential geometry. For a differential modality $\mathsf{S}$, its Eilenberg-Moore category (so its category of $\mathsf{S}$-algebras) is a tangent category \cite{cockett_et_al:LIPIcs:2020:11660}, while the opposite of its Kleisli category (so its category of free $\mathsf{S}$-algebras) is a cartesian differential category. As an application of our main result, we then get that $\mathsf{S}$-derivations form a tangent category (Cor \ref{cor:tan}), which means that ordinary derivations and $\mathcal{C}^\infty$-derivations each respectively form tangent categories (Ex \ref{ex:tan-poly} \& \ref{ex:tan-smooth}), and that $\mathsf{S}$-derivations on free $\mathsf{S}$-algebras form a cartesian differential category (Cor \ref{cor:CDC}). 

\smallskip

\textbf{Outline:} Sec \ref{sec:diffcats} and Sec \ref{sec:derivations} are background sections where we review differential categories and derivations. In Sec \ref{sec:monad} we show that a differential modality $\mathsf{S}$ lifts to a monad $\overline{\mathsf{S}}$ on the arrow category (Prop \ref{prop:monad}). Then, in Sec \ref{sec:deri-alg}, we show that the algebras of this monad $\overline{\mathsf{S}}$ correspond precisely to $\mathsf{S}$-derivations (Thm \ref{thm:deri=alg}), which is the main result of this paper. In Sec \ref{sec:der-monoid} we add biproducts to the story to define a monoidal structure on the arrow category (Lemma \ref{lem:arrow-monoidal}) and show that commutative monoids, in this case, correspond precisely to derivations in the classical sense (Prop \ref{prop:mon=der-arrow}). From this, we then show in Sec \ref{sec:diff-mod} that in the presence of biproducts, $\overline{\mathsf{S}}$ is in fact a differential modality, and thus the arrow category is itself a differential category (Thm \ref{thm:arrow-diff-cat}). Then in Sec \ref{sec:tan}, we briefly explain how we obtain a tangent category and a cartesian differential category of $\mathsf{S}$-derivations. 

\section{Differential Categories}\label{sec:diffcats}

In this background, mostly to introduce notation, we briefly review differential categories. For a more in-depth introduction on differential categories, we refer the reader to \cite{blute2006differential,Blute2019,blute2015derivations,lemay2019differential,o2017smoothness,blute2011kahler}.

\smallskip

\begin{remark} Before we begin we must stress an important point regarding terminology. In this paper, following recent conventions, by a differential category/modality, we mean it in the dual sense of \cite{blute2006differential,Blute2019}, so what was called a codifferential category/modality in \cite{blute2015derivations,lemay2019differential,o2017smoothness,blute2011kahler} for example. The original name ``differential categories" in \cite{blute2006differential} came from the fact that they were used for the categorical semantics of differential linear logic \cite{ehrhard2017introduction}, which meant they dealt with differential structure on \textit{co}monads and \textit{co}algebras. Thus as is customary in category theory, the dual notion ``codifferential categories" is then the proper setting to work with differential structure on monads and algebras. So the classical algebra notions of say derivations and differential algebras were then formalized in a codifferential category \cite{blute2015derivations,blute2011kahler,lemay2019differential}. This caused a bit of awkwardness, especially when trying to explain the story of differential categories to algebraist. Thus there has been a recent push to refer to codifferential categories instead as ``algebraic" differential categories or simply differential categories. In this paper we take this approach. 
\end{remark}

The underlying structure of a differential category is a symmetric monoidal category. We will denote an arbitrary category as $\mathbb{X}$, objects will be denoted using capital letters $A$, $B$, $C$ etc., homsets will be denoted as $\mathbb{X}(A,B)$ and maps will be denoted by minuscule letters $f,g,h, etc. \in \mathbb{X}(A,B)$. Arbitrary maps will be denoted using an arrow $f: A \to B$, identity maps as $1_A: A \to A$, and for composition we will use the standard notation $\circ$ (as opposed to diagrammatic notation which is often used in differential category literature, such as in \cite{blute2006differential,Blute2019,blute2015derivations,lemay2019differential,blute2011kahler}). Following the convention used in most of the literature on differential categories, we will work in a symmetric \emph{strict} monoidal category, so the associativity and unit isomorphisms for the monoidal product are equalities. For an arbitrary symmetric (strict) monoidal category $\mathbb{X}$, we denote its monoidal product as $\otimes$, the monoidal unit as $I$, and the natural symmetry isomorphism as $\sigma_{A, B}: A \otimes B \xrightarrow{\cong} B \otimes A$. Strictness allows us to write $A_1 \otimes A_2 \otimes \hdots \otimes A_n$ and $A \otimes I = A = I \otimes A$.

\smallskip

The underlying structure of a differential modality is a monad. Recall that a \textbf{monad} on a category $\mathbb{X}$ is a triple $(\mathsf{S}, \mu, \eta)$ consisting of an endofunctor $\mathsf{S}: \mathbb{X} \to \mathbb{X}$, a natural transformation $\mu_A: \mathsf{S}\mathsf{S}(A) \to \mathsf{S}(A)$, called the monad multiplication, and a natural transformation $\eta_A: A \to \mathsf{S}(A)$, called the monad unit, such that the following diagrams commute: 
\begin{equation}\begin{gathered}\label{diag:monad}\xymatrixrowsep{1.75pc}\xymatrixcolsep{5pc}\xymatrix{ 
        \mathsf{S}(A) \ar[r]^-{\mathsf{S}(\eta_A)} \ar[d]_-{\eta_{\mathsf{S}(A)}}   \ar@{=}[dr]^-{}& \mathsf{S} \mathsf{S}(A)  \ar[d]^-{\mu_A}  & \mathsf{S} \mathsf{S} \mathsf{S}(A) \ar[r]^-{\mathsf{S}(\mu_A)}  \ar[d]_-{\mu_{\mathsf{S}(A)}} & \mathsf{S}\mathsf{S}(A)  \ar[d]^-{\mu_A} \\
        \mathsf{S} \mathsf{S}(A) \ar[r]_-{\mu_A}  & \mathsf{S}(A)  & \mathsf{S} \mathsf{S}(A)  \ar[r]_-{\mu_A} & \mathsf{S}(A).  } \end{gathered}\end{equation}
        
        \smallskip
        
In a differential category, one asks that the free algebras of the monad be commutative monoids in a natural way. So recall that in a symmetric monoidal category $\mathbb{X}$, a \textbf{commutative monoid} is a triple $(A, \mathsf{m}, \mathsf{u})$ consisting of an object $A$, a map $\mathsf{m}: A \otimes A \to A$, called the multiplication, and a map $\mathsf{u}: I \to A$, called the unit, such that the following diagrams commute: 
\begin{equation}\begin{gathered}\label{diag:monoid}\xymatrixrowsep{1.75pc}\xymatrixcolsep{3pc}\xymatrix{A  \otimes  A  \otimes  A  \ar[r]^-{1_A \otimes \mathsf{m}} \ar[d]_-{\mathsf{m} \otimes 1_A}  & A  \otimes  A  \ar[d]^-{\mathsf{m}} & A \ar@{=}[dr] \ar[r]^-{1_A \otimes \mathsf{u}}\ar[d]_-{\mathsf{u} \otimes 1_A} & A  \otimes  A \ar[d]^-{\mathsf{m}} & A \otimes A  \ar[r]^-{\sigma_{A,A}}  \ar[dr]_-{\mathsf{m}} & A  \otimes  A \ar[d]^-{\mathsf{m}} \\
      A \otimes  A  \ar[r]_-{\mathsf{m}} &A  & A \otimes  A \ar[r]_-{\mathsf{m}}   & A  && A.} \end{gathered}\end{equation} 
We will also require that the monad multiplication be a monoid morphism. So recall that for commutative monoids $(A, \mathsf{m}, \mathsf{u})$ and $(A^\prime, \mathsf{m}^\prime, \mathsf{u}^\prime)$, a \textbf{monoid morphism} $f: (A, \mathsf{m}, \mathsf{u}) \to (A^\prime, \mathsf{m}^\prime, \mathsf{u}^\prime)$ between them is a map between the underlying objects $f: A \to A^\prime$, such that the following diagrams commute: 
\begin{equation}\begin{gathered}\label{diag:monoid-morph}\xymatrixrowsep{1.75pc}\xymatrixcolsep{5pc}\xymatrix{A \otimes A \ar[r]^-{f \otimes f}  \ar[d]_-{\mathsf{m}} & A^\prime \otimes A^\prime \ar[d]^-{\mathsf{m}^\prime} & I \ar[r]^-{\mathsf{u}} \ar[dr]_-{\mathsf{u}^\prime} & A \ar[d]^-{f}   \\
      A \ar[r]_-{f} & A^\prime & & A^\prime. }  \end{gathered}\end{equation}
We denote $\mathsf{CMON}[\mathbb{X}]$ to be category of commutative monoids of a symmetric monoidal category $\mathbb{X}$ and monoids morphisms between them. Then by an algebra modality, we mean a monad whose free algebras are naturally commutative monoids. 

\begin{definition} An \textbf{algebra modality} \cite[Dual of Def 1]{Blute2019} on symmetric monoidal category $\mathbb{X}$ is a quintuple $(\mathsf{S}, \mu, \eta, \mathsf{m}, \mathsf{u})$ consisting of a monad $(\mathsf{S}, \mu, \eta)$, a natural transformation $\mathsf{m}_A: \mathsf{S}(A) \otimes \mathsf{S}(A) \to \mathsf{S}(A)$, and a natural transformation $\mathsf{u}_A: I \to \mathsf{S}(A)$, such that for every object $A$: 
\begin{enumerate}[{\em (i)}]
\item $(\mathsf{S}(A), \mathsf{m}_A, \mathsf{u}_A)$ is a commutative monoid;
\item $\mu_A: (\mathsf{S}\mathsf{S}(A), \mathsf{m}_{\mathsf{S}(A)}, \mathsf{u}_{\mathsf{S}(A)}) \to (\mathsf{S}(A), \mathsf{m}_A, \mathsf{u}_A)$ is a monoid morphism. 
  \end{enumerate}
\end{definition}

It is worth mentioning that for an algebra modality, the naturality of $\mathsf{m}$ and $\mathsf{u}$ imply that for every map $f: A \to B$, $\mathsf{S}(f): (\mathsf{S}(A), \mathsf{m}_A, \mathsf{u}_A) \to (\mathsf{S}(B), \mathsf{m}_B, \mathsf{u}_B)$ is a monoid morphism. Therefore, the underlying endofunctor $\mathsf{S}$ factors through $\mathsf{CMON}[\mathbb{X}]$. In Sec \ref{sec:derivations} we will review how the algebras of an algebra modality are in fact also commutative monoids. 

\smallskip

A differential modality is an algebra modality that comes equipped with an extra natural transformation which behaves like a derivation. In order to express the Leibniz rule and the constant rule of differentiation, we need to be able to take the sum of maps and have zero maps. As such, we need to work in an additive symmetric monoidal category. Following the terminology in differential category literature \cite{blute2006differential,Blute2019}, by an \textbf{additive category} \cite[Def 3]{Blute2019} we mean a category $\mathbb{X}$ enriched over commutative monoids, that is, each homset $\mathbb{X}(A,B)$ is a commutative monoid, with binary operation $+$ and unit element $0: A \to B$, such that composition is a monoid morphism, that is, the following equalities holds for all suitable maps: 
\begin{align}
f \circ (g+h) = f\circ g + f \circ h && (g+h) \circ k = g \circ k + h \circ k && f \circ 0 = 0 && 0 \circ f = 0.
\end{align}
Note that here for an additive category, unlike in other references, we are not assuming the existence of biproducts or negatives. While negatives are not necessary for the story of this paper, we will add biproducts to our story later in Sec \ref{sec:deri-alg} and \ref{sec:diff-mod}.

\smallskip

An \textbf{additive symmetric monoidal category} \cite[Def 3]{Blute2019} is a symmetric monoidal category $\mathbb{X}$ which is also an additive category such that monoidal product preserves the additive structure, that is, the following equalities hold for all suitable maps: 
\begin{align}
f \otimes (g +h) = f \otimes g + f \otimes h && (g+h) \otimes k = g \otimes k + h \otimes k && f \otimes 0 = 0 && 0 \otimes f = 0.
\end{align}
With our additive structure, we are now in a position to properly define a differential modality and, with it, a differential category. 

\begin{definition}\label{def:diffcat} A \textbf{differential modality} \cite[Dual of Def 7]{Blute2019} on an additive symmetric monoidal category $\mathbb{X}$ is a sextuple $(\mathsf{S}, \mu, \eta, \mathsf{m}, \mathsf{u}, \mathsf{d})$ consisting of an algebra modality $(\mathsf{S}, \mu, \eta, \mathsf{m}, \mathsf{u})$ equipped with a natural transformation \[\mathsf{d}_A: \mathsf{S}(A) \to \mathsf{S}(A) \otimes A,\] called the \textbf{deriving transformation}, such that the following diagrams commute: 
  
  \begin{equation}\begin{gathered}\label{diag:deriving}
  \xymatrixrowsep{1.75pc}\xymatrixcolsep{2.75pc}\xymatrix{I \ar[r]^-{\mathsf{u}_A} \ar@/_/[dr]_-{0}  &\mathsf{S}(A)  \ar[d]^-{\mathsf{d}_A}  \ar@{}[dl]|(0.4){\text{\normalfont \textbf{[D.1]}}}  & \mathsf{S}(A) \otimes \mathsf{S}(A) \ar@{}[drr]|-{\text{\normalfont \textbf{[D.2]}}} \ar[rr]^-{(1_{\mathsf{S}(A)} \otimes \mathsf{d}_A) + (1_{\mathsf{S}(A)} \otimes \sigma_{A,\mathsf{S}(A)}) \circ (\mathsf{d}_A \otimes 1_{\mathsf{S}(A)})} \ar[d]_-{\mathsf{m}_A} && \mathsf{S}(A) \otimes \mathsf{S}(A) \otimes A \ar[d]^-{\mathsf{m}_A \otimes 1_A}   \\
  &  \mathsf{S}(A) \otimes A & \mathsf{S}(A) \ar[rr]_-{\mathsf{d}_A} && \mathsf{S}(A) \otimes A \\
 A \ar[r]^-{\eta_A} \ar@/_/[dr]_-{\mathsf{u}_A \otimes 1_A} & \mathsf{S}(A) \ar@{}[dl]|(0.4){\text{\normalfont \textbf{[D.3]}}}  \ar[d]^-{\mathsf{d}_A} & \mathsf{S}\mathsf{S}(A)\ar[rr]^-{\mu_A} \ar[d]_-{\mathsf{d}_{\mathsf{S}(A)}} \ar@{}[drr]|-{\text{\normalfont \textbf{[D.4]}}} && \mathsf{S}(A) \ar[d]^-{\mathsf{d}_A}   \\
& \mathsf{S}(A) \otimes A & \mathsf{S}\mathsf{S}(A) \otimes \mathsf{S}(A) \ar[r]_-{\mu_A \otimes \mathsf{d}_A} & \mathsf{S}(A) \otimes \mathsf{S}(A) \otimes A \ar[r]_-{\mathsf{m}_A \otimes 1_A} & \mathsf{S}(A) \otimes A \\
& \mathsf{S}(A)  \ar[d]_-{\mathsf{d}_A} \ar[rr]^-{\mathsf{d}_A} \ar@{}[drr]|-{\text{\normalfont \textbf{[D.5]}}} && \mathsf{S}(A) \otimes A \ar[d]^-{\mathsf{d}_A \otimes 1_A} \\
 & \mathsf{S}(A) \otimes A \ar[r]_-{\mathsf{d}_A \otimes 1_A} & \mathsf{S}(A) \otimes A \otimes A \ar[r]_-{1_{\mathsf{S}(A)} \otimes \sigma_{A,A}} & \mathsf{S}(A) \otimes A \otimes A. 
}
\end{gathered}\end{equation}
A \textbf{differential category} is an additive symmetric monoidal category equipped with a differential modality. 
\end{definition}

Let us give some brief intuitions about the axioms of a deriving transformation. For our differential modality, we can naively interpret $\mathsf{S}(A)$ as the object representing all ``smooth" functions from $A$ to $I$. The monad multiplication $\mu$ captures composition of these smooth functions, while the monad $\eta$ picks out the linear smooth functions and, in particular, the identity function. The monoid multiplication $\mathsf{m}$ says that we can multiply smooth functions together, while the monoid unit $\mathsf{u}$ picks out the constant functions. The deriving transformation $\mathsf{d}$ of course captures differentiating smooth functions, where the idea is that it sends a smooth function $f$ to its derivative $f^\prime \otimes \mathsf{d}(x)$. Then \textbf{[D.1]} is the constant rule which says that the derivative of a constant is zero; \textbf{[D.2]} is the famous Leibniz rule which tells us how to differentiate a product; \textbf{[D.3]} is the linear rule which says that the derivative of a linear function is a constant; \textbf{[D.4]} is the equally famous chain rule which tells us how to differentiation a composite of functions; and lastly \textbf{[D.5]} is the interchange rule which captures the symmetry of the mixed partial derivatives. This intuition about interpreting differential modalities in terms of smooth functions of some sort can be made precise in the context of differential linear logic -- we invite the curious reader to see \cite{ehrhard2017introduction}. It is worth noting that \textbf{[D.5]} was not included in the original definition \cite[Def 2.5]{blute2006differential} but has since been added to definition of deriving transformation -- see \cite[Sec 4]{Blute2019} for details. 

\smallskip

We now give our two main examples of differential categories, one that captures polynomial differentiation and one that captures differentiating real smooth functions. For other interesting examples of differential categories, we invite the reader to see \cite{Blute2019}.

\begin{example} \textbf{\textsf{Polynomials.}} \label{ex:sym} Let $\mathbb{K}$ be a field and let $\mathsf{VEC}_{\mathbb{K}}$ be the category of $\mathbb{K}$-vector spaces and $\mathbb{K}$-linear maps between them. Then $\mathsf{VEC}_{\mathbb{K}}$ is a differential category where symmetric algebras \cite[Chap 14, Sec 8]{lang2002algebra} induce a differential modality which captures differentiation polynomials. Indeed, recall that for a $\mathbb{K}$-vector space $V$, its symmetric algebra $\mathsf{Sym}(V)$ is the free commutative $\mathbb{K}$-algebra over $V$. If $X$ is a basis of $V$, then there is a canonical isomorphism of $\mathbb{K}$-algebras  $\mathsf{Sym}(V) \cong \mathbb{K}[X]$, where $\mathbb{K}[X]$ is the polynomial ring over $X$. We then obtain a differential modality $\mathsf{Sym}$ on $\mathsf{VEC}_{\mathbb{K}}$ where the monad structure corresponds to polynomial composition, the monoid structure corresponds to polynomial multiplication, and the deriving transformation (written in terms of polynomial rings) $\mathsf{d}_V: \mathbb{K}[X] \to \mathbb{K}[X] \otimes V$ captures differentiating polynomials by sending a polynomial $p \in \mathbb{K}[X]$ to the sum of its partial derivatives: 
\[ \mathsf{d}_V(p) = \sum \limits_{x \in X} \frac{\partial p}{\partial x} \otimes x. \]
 For full details on this differential modality, we invite the reader to see \cite[Sec 2.5.3]{blute2006differential}.
\end{example}

It is worth pointing out that the above example easily generalizes to the category of modules over a commutative (semi)ring \cite[Ex 1]{Blute2019}. In fact, in any additive symmetric monoidal category with enough colimits to build the free symmetric algebra is going to be a differential category \cite[Thm 6.1]{blute2015derivations}. For example, the category of sets and relations is a differential category in this manner \cite[Sec 2.5.1]{blute2006differential}. 

\begin{example}\label{ex:s-inf} \textbf{\textsf{Real smooth functions.}} Let $\mathbb{K} = \mathbb{R}$ the field of real numbers. Then $\mathsf{VEC}_{\mathbb{R}}$ admits another differential modality which this time captures differentiating smooth functions by instead considering $\mathcal{C}^\infty$-rings \cite[Chap 1]{moerdijk2013models}. Briefly, recall that a $\mathcal{C}^\infty$-ring is a commutative $\mathbb{R}$-algebra $A$ such that for every real smooth function $f: \mathbb{R}^n \to \mathbb{R}$ there is a function $\Phi_f: A^n \to A$ satisfying certain coherences with respect to composition and projections. The intuition here is that $\Phi_f$ can be interpreted as evaluating the smooth function $f$ with elements of $A$. The main class of examples of $\mathcal{C}^\infty$-rings comes from smooth manifolds \cite[Chap 2]{moerdijk2013models}, where, for a smooth manifold $M$, \[\mathcal{C}^\infty(M) = \lbrace f: M \to \mathbb{R} \vert~ \text{$f$ is a real smooth function} \rbrace\] is a $\mathcal{C}^\infty$-ring. In particular, for every $\mathbb{R}$-vector space $V$, there exits a free $\mathcal{C}^\infty$-ring $\mathsf{S}^\infty(V)$ over it \cite[Sec 4]{cruttwell2019integral}, where for $V = \mathbb{R}^n$ we have that $\mathsf{S}^\infty(\mathbb{R}^n) = \mathcal{C}^\infty(\mathbb{R}^n)$. In turn, this induces a differential modality $\mathsf{S}^\infty$ on $\mathsf{VEC}_{\mathbb{R}}$ where, as expected, the monad structure corresponds to composing smooth functions, the monoid structure corresponds to multiplying smooth functions, and the deriving transformation captures differentiating smooth functions. In the special case of $V = \mathbb{R}^n$, the deriving transformation $\mathsf{d}_{\mathbb{R}^n}: \mathcal{C}^\infty(\mathbb{R}^n) \to \mathcal{C}^\infty(\mathbb{R}^n) \otimes \mathbb{R}^n$ sends a smooth function $f: \mathbb{R}^n \to \mathbb{R}$ to the sum of its partial derivatives: 
\[ \mathsf{d}_{\mathbb{R}^n}(f) = \sum \limits_{i=1}^{n} \frac{\partial f}{\partial x_i} \otimes \vec{e_i} \]
where $\vec{e_i} \in \mathbb{R}^n$ is the canonical basis vector with $1$ in the $i$-th coordinate and zeroes everywhere else. For full details on this differential modality, we invite the reader to see \cite[Sec 5]{cruttwell2019integral}. 
\end{example}

\section{Derivations}\label{sec:derivations}

In this section we briefly review derivations in a differential category. For a more detailed introduction, we invite the reader to see \cite{blute2015derivations,o2017smoothness,lemay2019differential,blute2011kahler}.

\smallskip

In any additive symmetric monoidal category $\mathbb{X}$, one can already generalize the notion of a derivation from classical algebra to derivations for commutative monoids \cite[Def 4.1]{blute2011kahler}. To do so, we must first review modules. Recall that for a commutative monoid $(A, \mathsf{m}, \mathsf{u})$, an \textbf{$(A, \mathsf{m}, \mathsf{u})$-module} is a pair $(M, \alpha)$ consisting of an object $A$ and a map ${\alpha: A \otimes M \to M}$, called the action, such that the following diagrams commute: 
\begin{equation}\begin{gathered}\label{diag:module}\xymatrixrowsep{1.75pc}\xymatrixcolsep{5pc}\xymatrix{A  \otimes  A  \otimes  M  \ar[r]^-{1_A \otimes \alpha} \ar[d]_-{\mathsf{m} \otimes 1_M}  & A  \otimes  M \ar[d]^-{\alpha} & M \ar@{=}[dr] \ar[r]^-{\mathsf{u} \otimes 1_M} & A  \otimes  M \ar[d]^-{\alpha}  \\
      A \otimes  M  \ar[r]_-{\alpha} & M  &  & M.} \end{gathered}\end{equation} 
Then a \textbf{derivation} of $(A, \mathsf{m}, \mathsf{u})$ evaluated in $(M, \alpha)$ is a map $\mathsf{D}: A \to M$ such that the following diagrams commute: 
\begin{equation}\begin{gathered}\label{diag:derivation-normal}
 \xymatrixrowsep{1.75pc}\xymatrixcolsep{2.75pc}\xymatrix{I \ar[r]^-{\mathsf{u}} \ar[dr]_-{0}  & A  \ar[d]^-{\mathsf{D}}  & A \otimes A \ar[rr]^-{(1_{A} \otimes \mathsf{D}) + \sigma_{M,A} \circ (\mathsf{D} \otimes 1_A)} \ar[d]_-{\mathsf{m}} && A \otimes M \ar[d]^-{\alpha}   \\
  & M & A \ar[rr]_-{\mathsf{D}} && M.
}
\end{gathered}\end{equation}
As a shorthand, we will denote derivations as $\mathsf{D}: (A, \mathsf{m}, \mathsf{u}) \to (M,\alpha)$. Now the right diagram above is the Leibniz rule, while the diagram on the left is the constant rule which says the derivative of the units is zero. Readers familiar with derivations from classical algebra will note that the diagram on the left is usually not part of the definition and can be deduced \cite[Chap 19, Sec 3]{lang2002algebra}. However, the proof requires negatives, and since we are not assuming negatives, it must be included in the definition. 

Now for derivations $\mathsf{D}: (A, \mathsf{m}, \mathsf{u}) \to (M,\alpha)$ and $\mathsf{D}^\prime: (A^\prime, \mathsf{m}^\prime, \mathsf{u}^\prime) \to (M^\prime,\alpha^\prime)$, a \textbf{derivation morphism} is a pair $(f,g): \left( \mathsf{D}: (A, \mathsf{m}, \mathsf{u}) \to (M,\alpha)\right) \to \left(\mathsf{D}^\prime: (A^\prime, \mathsf{m}^\prime, \mathsf{u}^\prime) \to (M^\prime,\alpha^\prime) \right)$ where $f: (A, \mathsf{m}, \mathsf{u}) \to (A^\prime, \mathsf{m}^\prime, \mathsf{u}^\prime)$ is a monoid morphism and $g: M \to M^\prime$ is a map such that the following diagrams commute: 

\begin{equation}\begin{gathered}\label{diag:derivation-morph}\xymatrixrowsep{1.75pc}\xymatrixcolsep{5pc}\xymatrix{A \otimes M \ar[r]^-{f \otimes g}  \ar[d]_-{\alpha} & A^\prime \otimes M^\prime \ar[d]^-{\alpha^\prime} & A \ar[r]^-{f} \ar[d]_-{\mathsf{D}} & A^\prime \ar[d]^-{\mathsf{D}^\prime}   \\
      M \ar[r]_-{g} & M^\prime & M \ar[r]_-{g} & M^\prime. }  \end{gathered}\end{equation}
Note that the diagram on the left essentially says that $g$ is a module morphism. We denote by $\mathsf{DER}[\mathbb{X}]$ the category of derivations in our additive symmetric monoidal category $\mathbb{X}$ and derivations morphisms between them. 

\smallskip

In a differential category, we can actually improve upon the previous definition and define derivations relative to a differential modality. As such, instead of simply considering derivations for commutative monoids, we may consider derivations for algebras (in the monad sense) of our differential modality.

\smallskip

First recall that for a monad $(\mathsf{S}, \mu, \eta)$, an \textbf{$\mathsf{S}$-algebra} is a pair $(A, \nu)$ consisting of an object $A$ and a map ${\nu: \mathsf{S}(A) \to A}$, called the $\mathsf{S}$-algebra structure map, such that the following diagrams commute: 
\begin{equation}\begin{gathered}\label{diag:S-alg}\xymatrixrowsep{1.75pc}\xymatrixcolsep{5pc}\xymatrix{ 
        A \ar[r]^-{\eta_A} \ar@{=}[dr]^-{}& \mathsf{S}(A)  \ar[d]^-{\nu}  & \mathsf{S}  \mathsf{S}(A) \ar[r]^-{\mathsf{S}(\nu)}  \ar[d]_-{\mu_A} & \mathsf{S}(A)  \ar[d]^-{\nu} \\
       & A & \mathsf{S}(A)  \ar[r]_-{\nu} & A. } \end{gathered}\end{equation}
Also recall that for $\mathsf{S}$-algebras $(A, \nu)$ and $(A^\prime, \nu^\prime)$, an $\mathsf{S}$-algebra morphism $f: (A, \nu) \to (A^\prime, \nu^\prime)$ is a map between the underlying objects $f: A \to A^\prime$ such that the following diagram commutes: 
\begin{equation}\begin{gathered}\label{diag:SAlg-morph}\xymatrixrowsep{1.75pc}\xymatrixcolsep{5pc}\xymatrix{\mathsf{S}(A) \ar[r]^-{\mathsf{S}(f)}  \ar[d]_-{\nu} & \mathsf{S}(A^\prime)  \ar[d]^-{\nu^\prime} \\
      A \ar[r]_-{f} & A^\prime.  }  \end{gathered}\end{equation}
For a monad $(\mathsf{S}, \mu, \eta)$, we denote by $\mathsf{S}\text{-}\mathsf{ALG}$ the category of $\mathsf{S}$-algebras and $\mathsf{S}$-algebra morphisms between them, also often called the Eilenberg-Moore category of the monad $(\mathsf{S}, \mu, \eta)$. 

\smallskip

It turns out that for an algebra modality $(\mathsf{S}, \mu, \eta, \mathsf{m}, \mathsf{u})$, the $\mathsf{S}$-algebras are canonically commutative monoids \cite[Thm 3.12]{blute2015derivations}. Indeed, given an $\mathsf{S}$-algebra $(A, \nu)$, define the maps $\mathsf{m}^\nu: A \otimes A \to A$ and $\mathsf{u}^\nu: I \to A$ as the following composites: 
\begin{align}\label{def:Salg-m-u}
\xymatrixcolsep{3pc}\xymatrix{ \mathsf{m}^\nu: A \otimes A \ar[r]^-{\eta_A \otimes \eta_A} & \mathsf{S}(A) \otimes  \mathsf{S}(A) \ar[r]^-{\mathsf{m}_A} & \mathsf{S}(A) \ar[r]^-{\nu} & A 
  }  && \xymatrixcolsep{3pc}\xymatrix{ \mathsf{u}^\nu: I \ar[r]^-{\mathsf{u}_A} & \mathsf{S}(A) \ar[r]^-{\nu} & A. 
  }
\end{align}
Then $(A, \mathsf{m}^\nu, \mathsf{u}^\nu)$ is a commutative monoid. Observe that applying this construction to a free $\mathsf{S}$-algebra $(\mathsf{S}(A), \mu_A)$ results precisely to the natural monoid structure from that algebra modality, that is, $\mathsf{m}^{\mu_A} = \mathsf{m}_A$ and $\mathsf{u}^{\mu_A} = \mathsf{u}_A$. Moreover, if ${f: (A, \nu) \to (A^\prime, \nu^\prime)}$ is an $\mathsf{S}$-algebra morphism, then $f: (A, \mathsf{m}^\nu, \mathsf{u}^\nu) \to (A^\prime, \mathsf{m}^{\nu^\prime}, \mathsf{u}^{\nu^\prime})$ is a monoid morphism. 

\smallskip

To define derivations relative to a differential modality, we must also consider modules of our $\mathsf{S}$-algebras. So for an $\mathsf{S}$-algebra $(A, \nu)$, by an $(A, \nu)$-module, we simply mean an $(A, \mathsf{m}^\nu, \mathsf{u}^\nu)$-module. 

\begin{definition}\label{def:S-der} Let $(\mathsf{S}, \mu, \eta, \mathsf{m}, \mathsf{u}, \mathsf{d})$ be a differential modality, $(A,\nu)$ be an $\mathsf{S}$-algebra, and $(M,\alpha)$ be an $(A, \nu)$-module. An \textbf{$\mathsf{S}$-derivation} \cite[Def 5.12]{blute2015derivations} of $(A,\nu)$ evaluated in $(M,\alpha)$ is a map $\mathsf{D}: A \to M$ such that the following diagram commutes: 
  \begin{equation}\begin{gathered}\label{diag:derivation}
  \xymatrixrowsep{1.75pc}\xymatrixcolsep{5pc}\xymatrix{
\mathsf{S}(A)\ar[rr]^-{\nu} \ar[d]_-{\mathsf{d}_A}  && A \ar[d]^-{\mathsf{D}}   \\
 \mathsf{S}(A) \otimes A \ar[r]_-{\nu \otimes \mathsf{D}} & A \otimes M \ar[r]_-{\alpha} & M.  
}
\end{gathered}\end{equation}
As a shorthand, we will denote $\mathsf{S}$-derivations as $\mathsf{D}: (A,\nu) \to (M,\alpha)$. For $\mathsf{S}$-derivations $\mathsf{D}: (A,\nu) \to (M,\alpha)$ and $\mathsf{D}^\prime:(A^\prime,\nu^\prime) \to (M^\prime,\alpha^\prime)$ an \textbf{$\mathsf{S}$-derivation morphism} between them is a pair \[(f,g): \left(\mathsf{D}: (A,\nu) \to (M,\alpha) \right) \to \left( \mathsf{D}^\prime: (A^\prime,\nu^\prime) \to (M^\prime,\alpha^\prime) \right)\] where $f: (A,\nu) \to (A^\prime, \nu^\prime)$ is an $\mathsf{S}$-algebra morphism and $g: M \to M^\prime$ is a map such that the diagrams in (\ref{diag:derivation-morph}) commute. We denote by $\mathsf{S}\text{-}\mathsf{DER}$ the category of $\mathsf{S}$-derivations and $\mathsf{S}$-derivation morphisms between them. 
\end{definition}

One should interpret (\ref{diag:derivation}) as a chain rule axiom. Indeed, taking our previous interpretation of $\mathsf{S}(A)$ as an algebra of smooth functions $f: A \to I$, then (\ref{diag:derivation}) says that $\mathsf{D}(\nu(f))= \nu(f^\prime) \mathsf{D}(x)$. Moreover, every $\mathsf{S}$-derivation $\mathsf{D}: (A,\nu) \to (M,\alpha)$ is in fact also a derivation $\mathsf{D}: (A, \mathsf{m}, \mathsf{u}) \to (M,\alpha)$ \cite[Prop 5.2]{lemay2019differential} and, of course, every $\mathsf{S}$-derivation morphism is a derivation morphism. In summary, while ordinary derivations are axiomatized by the Leibniz rule, $\mathsf{S}$-derivations are axiomatized by the chain rule and are themselves ordinary derivations as well. 

\smallskip

The canonical example of an $\mathsf{S}$-derivation is the deriving transformation itself \cite[Thm 5.13]{blute2015derivations}. Indeed, note that $(\mathsf{S}(A) \otimes A, \mathsf{m}_A \otimes 1_A)$ is an $(\mathsf{S}(A), \mu_A)$-module, then the chain rule axiom \textbf{[D.4]} says precisely that \[\mathsf{d}_A: (\mathsf{S}(A), \mu_A) \to (\mathsf{S}(A) \otimes A, \mathsf{m}_A \otimes 1_A)\] is an $\mathsf{S}$-derivation. In fact, it is the \textit{universal} $\mathsf{S}$-derivation for $(\mathsf{S}(A), \mu_A)$, in the sense that every $\mathsf{S}$-derivation on $(\mathsf{S}(A), \mu_A)$ factors through it \cite[Thm 5.15]{blute2015derivations}.
\smallskip

We can also consider universal $\mathsf{S}$-derivations for arbitrary $\mathsf{S}$-algebras \cite[Def 5.14]{blute2015derivations}, which generalize the notion of a module of Kahler differentials \cite[Chap 19, Sec 3]{lang2002algebra} for $\mathsf{S}$-algebras. As such, for free $\mathsf{S}$-algebras $(\mathsf{S}(A), \mu_A)$, one could say that $(\mathsf{S}(A) \otimes A, \mathsf{m}_A \otimes 1_A)$ is its module of Kahler $\mathsf{S}$-differentials. However for an arbitrary $\mathsf{S}$-algebra $(A,\nu)$, such a module of Kahler $\mathsf{S}$-differentials may not exist. That said, under mild coequalizer assumptions, it is possible to construct universal $\mathsf{S}$-derivations for any $\mathsf{S}$-algebra \cite[Thm 5.23]{blute2015derivations}. Universal $\mathsf{S}$-derivations will not play role in the story of this paper, so we invite the curious reader to see \cite{blute2015derivations,o2017smoothness,blute2011kahler} to learn more about them. 
\smallskip

On the other hand, $\mathsf{S}$-derivations can also be used to generalize differential algebras in a differential category, see \cite{lemay2019differential}. 

\smallskip
We conclude this section with $\mathsf{S}$-derivations in our two main examples of differential categories. 

\begin{example}\label{ex:der-sym} \textbf{\textsf{Ordinary derivations.}} For the symmetric algebra differential modality $\mathsf{Sym}$ on $\mathbb{VEC}_\mathbb{K}$, $\mathsf{Sym}$-derivations correspond precisely to ordinary derivations \cite[Remark 5.8]{blute2015derivations}. Indeed, first note that $\mathsf{Sym}$-algebras are precisely commutative $\mathbb{K}$-algebras. Then recall that for a commutative $\mathbb{R}$-algebra $A$ and an $A$-module $M$, a derivation is a $\mathbb{K}$-linear map $\mathsf{D}: A \to M$ which satisfies the Leibniz rule:
\[\mathsf{D}(ab) = a \mathsf{D}(b) + \mathsf{D}(a)b.\] 
As we have negatives, we also have that $\mathsf{D}(1) =0$, and so these are the same as derivations in our additive symmetric monoidal category $\mathsf{VEC}_\mathbb{K}$. As explained above, every $\mathsf{Sym}$-derivation is going to be a derivation. Of course, it is natural to ask how every derivation can also be a $\mathsf{Sym}$-derivation, given that an ordinary derivation is only required to satisfy the Leibniz rule and not a chain rule. It turns out that derivations do in fact satisfy a chain rule with respect to polynomials.
Indeed, given a polynomial $p \in \mathbb{K}[x_1, \hdots, x_n]$, by the Leibniz rule we get that the following equality holds for all $a_i \in A$: 
\[ \mathsf{D}\left( p(a_1, \hdots, a_n \right) = \sum \limits_{i=1}^n \frac{\partial p}{\partial x}(a_1, \hdots, a_n) \mathsf{D}(a_i). \] 
It is not difficult to see that this tells us precisely that a derivation satisfies (\ref{diag:derivation}), and is thus a $\mathsf{Sym}$-derivation. 
\end{example}

\begin{example}\label{ex:der-smooth} \textbf{\textsf{$\mathcal{C}^\infty$-derivations.}} For the free $\mathcal{C}^\infty$-ring modality $\mathsf{S}^\infty$ on $\mathbb{VEC}_\mathbb{R}$, $\mathsf{S}^\infty$-derivations correspond precisely to $\mathcal{C}^\infty$-derivations \cite[Thm 5.25]{cruttwell2019integral}. Indeed, as expected, $\mathsf{S}^\infty$-algebras are precisely $\mathcal{C}^\infty$-rings, and so $\mathsf{S}^\infty$-derivations should be the appropriate generalization of derivations for $\mathcal{C}^\infty$-rings. The notion of $\mathcal{C}^\infty$-derivations was first introduced by Kock and Dubuc in their study of derivations for Fermat theories \cite{dubuc19841}, and then considered more explicitly by Joyce in their study of algebraic geometry over $\mathcal{C}^\infty$-rings \cite{joyce2011introduction}. So for a $\mathcal{C}^\infty$-ring $A$ and a module $M$ over the underlying $\mathbb{R}$-algebra $A$, a \textbf{$\mathcal{C}^\infty$-derivation} \cite[Def 2.15]{joyce2011introduction} is an $\mathbb{R}$-linear map $\mathsf{D}: A \to M$ such that for every smooth function $f: \mathbb{R}^n \to \mathbb{R}$ and all $a_i \in A$, the following equality holds: 
\[ \mathsf{D}\left(\Phi_f(a_1, \hdots, a_n)\right) = \sum \limits_{i=1}^{n} \frac{\partial f}{\partial x_i}(a_1, \hdots,a_n)\mathsf{D}(a_i). \]
Intuitively, the above identity says that a $\mathcal{C}^\infty$-derivation satisfies the chain rule with respect to smooth functions. It then follows that $\mathcal{C}^\infty$-derivations and $\mathsf{S}^\infty$-derivations are indeed the same -- see \cite[Sec 5.1]{cruttwell2019integral} for more full details. 
\end{example}

\section{Monad on Arrow Category}\label{sec:monad}

In this section, we show that a differential modality lifts to a monad on the arrow category, and then, in the next section, we show that the algebras of this monad are precisely  derivations. So for the remainder of this section, let $\mathbb{X}$ be a differential category with differential modality $(\mathsf{S}, \mu, \eta, \mathsf{m}, \mathsf{u})$. 

Recall that for our underlying category $\mathbb{X}$, its \textbf{arrow category} is the category $\mathsf{Arr}[\mathbb{X}]$ whose objects are maps $\phi: A_0 \to A_1$ of $\mathbb{X}$, and whose maps are pairs $(f_0, f_1): \left( \phi: A_0 \to A_1 \right) \to \left(\phi^\prime: A_0^\prime \to A_1^\prime\right)$ of maps $f_0: A_0 \to A_0^\prime$ and $f_1: A_1 \to A_1^\prime$ of $\mathbb{X}$ such that the following diagram commutes: 
\begin{equation}\begin{gathered}\label{diag:arrow-morph} \xymatrixcolsep{5pc}\xymatrix{
        A_0 \ar[r]^-{\phi} \ar[d]_-{f_0} & A_1 \ar[d]^-{f_1}  \\
    A_0^\prime \ar[r]_-{\phi^\prime} & A_1^\prime.} \end{gathered}\end{equation}
Composition in $\mathsf{Arr}[\mathbb{X}]$ is defined pointwise as in $\mathbb{X}$, that is, $(g_0, g_1) \circ (f_0, f_1) = (g_0 \circ f_0, g_1 \circ f_1)$, while identities in $\mathsf{Arr}[\mathbb{X}]$ are pairs of identities in $\mathbb{X}$, that is, $1_{\left(\phi: A_0 \to A_1 \right)} = (1_{A_0}, 1_{A_1})$. 

We will now define a monad on the arrow category $\mathsf{Arr}[\mathbb{X}]$. First define the endofunctor \[\overline{\mathsf{S}}: \mathsf{Arr}[\mathbb{X}] \to \mathsf{Arr}[\mathbb{X}]\] on objects as follows: 

\begin{align}\label{diag:overlineS-def}
\overline{\mathsf{S}}\left(\phi: A_0 \to A_1 \right):= \left( \xymatrixcolsep{3pc}\xymatrix{  \mathsf{S}(A_0) \ar[r]^-{\mathsf{d}_{A_0}} & \mathsf{S}(A_0) \otimes A_0 \ar[r]^-{1_{\mathsf{S}(A_0)} \otimes \phi} & \mathsf{S}(A_0) \otimes A_1 
  }\right)
\end{align}
and on maps as $\overline{S}(f_0, f_1) := (\mathsf{S}(f_0),\mathsf{S}(f_0) \otimes f_1)$. 

\begin{remark} For readability, in proofs throughout the rest of the paper, we will often omit subscripts. \end{remark}

\begin{lemma} $\overline{\mathsf{S}}: \mathsf{Arr}[\mathbb{X}] \to \mathsf{Arr}[\mathbb{X}]$ is a functor.
\end{lemma}
\begin{proof} We first need to explain why $\overline{S}$ is well-defined on maps. So let $(f_0, f_1): \left( \phi: A_0 \to A_1 \right) \to \left(\phi^\prime: A_0^\prime \to A_1^\prime\right)$ be a map in $\mathsf{Arr}[\mathbb{X}]$. By naturality of the deriving transformation, and since $(f_0, f_1)$ is a map in the arrow category, the following diagram commutes:
\[ \xymatrixcolsep{5pc}\xymatrix{ \mathsf{S}(A_0) \ar[d]_-{\mathsf{S}(f_0)} \ar[r]^-{\mathsf{d}} \ar@{}[dr]|-{\text{\normalfont Nat of $\mathsf{d}$}} & \mathsf{S}(A_0) \otimes A_0 \ar[r]^-{1 \otimes \phi} \ar[d]|-{\mathsf{S}(f_0) \otimes f_0} \ar@{}[dr]|-{\text{\normalfont (\ref{diag:arrow-morph})}}  & \mathsf{S}(A_0) \otimes A_1 \ar[d]^-{\mathsf{S}(f_0) \otimes f_1} \\ 
\mathsf{S}(A_0^\prime) \ar[r]_-{\mathsf{d}} & \mathsf{S}(A_0^\prime) \otimes A_0^\prime \ar[r]_-{1 \otimes \phi^\prime} & \mathsf{S}(A_0^\prime) \otimes A_1^\prime.
  } \]
  Therefore $\overline{S}(f_0, f_1): \overline{\mathsf{S}}\left(\phi: A_0 \to A_1 \right) \to \overline{\mathsf{S}}\left(\phi^\prime: A_0^\prime \to A_1^\prime\right)$ satisfies (\ref{diag:arrow-morph}) and is a map in the arrow category. As such $\overline{\mathsf{S}}$ is well-defined. Functoriality of $\overline{\mathsf{S}}$ immediately follows from functoriality of $\mathsf{S}$. 
\end{proof}

Now we define our monad unit $\overline{\eta}$. So define \[\overline{\eta}_{\left(\phi: A_0 \to A_1 \right)}: \left( \phi: A_0 \to A_1 \right) \to \overline{\mathsf{S}}\left( \phi: A_0 \to A_1 \right)\] as follows: 
\begin{align}
\overline{\eta}_{\left(\phi: A_0 \to A_1 \right)} := \left( \eta_{A_0}: A_0 \to \mathsf{S}(A_0), \mathsf{u}_{A_0} \otimes 1_{A_1}: A_1 \to \mathsf{S}(A_0) \otimes A_1 \right).
\end{align}

\begin{lemma} $\overline{\eta}$ is a natural transformation. 
\end{lemma}
\begin{proof} We first need to show that $\overline{\eta}$ is well-defined. To do so, note that by the linear rule \textbf{[D.3]}, the following diagram commutes: 
\[ \xymatrixcolsep{5pc}\xymatrix{ A_0 \ar[d]_-{\eta} \ar[rr]^-{\phi} \ar[dr]^-{\mathsf{u} \otimes 1} && A_1 \ar[d]^-{\mathsf{u} \otimes 1}  \\ 
\mathsf{S}(A_0)  \ar@{}[ur]|(0.3){\text{\normalfont \textbf{[D.3]}}}  \ar[r]_-{\mathsf{d}} & \mathsf{S}(A_0) \otimes A_0 \ar[r]_-{1 \otimes \phi} & \mathsf{S}(A_0) \otimes A_1.
  } \]
  Therefore $\overline{\eta}_{\left(\phi: A_0 \to A_1 \right)}: \left( \phi: A_0 \to A_1 \right) \to \overline{\mathsf{S}}\left( \phi: A_0 \to A_1 \right)$ satisfies (\ref{diag:arrow-morph}) and is indeed a map in the arrow category. So $\overline{\eta}$ is well-defined. Naturality of $\overline{\eta}$ follows immediately from naturality of $\eta$ and $\mathsf{u}$. 
\end{proof}

To define the monad multiplication $\overline{\mu}$, we first expand $\overline{\mathsf{S}}\overline{\mathsf{S}}\left(\phi: A_0 \to A_1 \right)$ to be given by the following composite: 
\begin{align*}
\xymatrixcolsep{4.5pc}\xymatrix{  \mathsf{S}\mathsf{S}(A_0) \ar[r]^-{\mathsf{d}_{\mathsf{S}(A_0)}} & \mathsf{S}\mathsf{S}(A_0) \otimes \mathsf{S}(A_0) \ar[r]^-{1_{\mathsf{S}(A_0)} \otimes \mathsf{d}_{A_0}} & \mathsf{S}\mathsf{S}(A_0) \otimes \mathsf{S}(A_0) \otimes A_0 \ar[r]^-{1_{\mathsf{S}\mathsf{S}(A_0)} \otimes 1_{\mathsf{S}(A_0)} \otimes \phi} & \mathsf{S}\mathsf{S}(A_0) \otimes \mathsf{S}(A_0) \otimes A_1.
  }
\end{align*}
Then define \[\overline{\mu}_{\left(\phi: A_0 \to A_1 \right)}: \overline{\mathsf{S}}~\overline{\mathsf{S}}\left( \phi: A_0 \to A_1 \right) \to \overline{\mathsf{S}}\left( \phi: A_0 \to A_1 \right)\] as the pair of following maps: 
\begin{equation}\begin{gathered}\label{diag:overlinemu} 
\xymatrixcolsep{5pc}\xymatrix{  \mathsf{S}\mathsf{S}(A_0) \ar[r]^-{\mu_{A_0}} & \mathsf{S}(A_0)}  \\ 
\xymatrixcolsep{5pc}\xymatrix{  \mathsf{S}\mathsf{S}(A_0) \otimes \mathsf{S}(A_0) \otimes A_1 \ar[r]^-{\mu_{A_0} \otimes  1_{\mathsf{S}(A_0)} \otimes 1_{A_1}} & \mathsf{S}(A_0) \otimes \mathsf{S}(A_0) \otimes A_1 \ar[r]^-{\mathsf{m}_{A_0}\otimes 1_{A_1}} & \mathsf{S}(A_0) \otimes A_1. 
  }
\end{gathered}\end{equation}

\begin{lemma} $\overline{\mu}$ is a natural transformation. 
\end{lemma}
\begin{proof} We first need to show that $\overline{\mu}$ is well-defined. To do so, note that by the chain rule \textbf{[D.4]}, the following diagram commutes: 
\[ \xymatrixcolsep{4.5pc}\xymatrix{  \mathsf{S}\mathsf{S}(A_0) \ar@{}[ddrr]|-{\text{\normalfont \textbf{[D.4]}}} \ar[dd]_-{\mu} \ar[r]^-{\mathsf{d}} & \mathsf{S}\mathsf{S}(A_0) \otimes \mathsf{S}(A_0) \ar[r]^-{1 \otimes \mathsf{d}} & \mathsf{S}\mathsf{S}(A_0) \otimes \mathsf{S}(A_0) \otimes A_0 \ar[r]^-{1 \otimes 1 \otimes \phi} \ar[d]^-{\mu \otimes  1 \otimes 1} & \mathsf{S}\mathsf{S}(A_0) \otimes \mathsf{S}(A_0) \otimes A_1 \ar[d]^-{\mu \otimes  1 \otimes 1} \\ 
& & \mathsf{S}(A_0) \otimes \mathsf{S}(A_0) \otimes A_0 \ar[d]^-{\mathsf{m} \otimes 1} & \mathsf{S}(A_0) \otimes \mathsf{S}(A_0) \otimes A_1 \ar[d]^-{\mathsf{m}\otimes 1}
\\
\mathsf{S}(A_0)    \ar[rr]_-{\mathsf{d}_{A_0}} && \mathsf{S}(A_0) \otimes A_0 \ar[r]_-{1 \otimes \phi} & \mathsf{S}(A_0) \otimes A_1.
  } \]
  Therefore $\overline{\mu}_{\left(\phi: A_0 \to A_1 \right)}: \overline{\mathsf{S}}~\overline{\mathsf{S}}\left( \phi: A_0 \to A_1 \right) \to \overline{\mathsf{S}}\left( \phi: A_0 \to A_1 \right)$ satisfies (\ref{diag:arrow-morph}) and is indeed a map in the arrow category. So $\overline{\mu}$ is well-defined. Naturality of $\overline{\mu}$ follows immediately from naturality of $\mu$ and $\mathsf{m}$. 
\end{proof}

\begin{proposition}\label{prop:monad} $(\overline{\mathsf{S}}, \overline{\mu}, \overline{\eta})$ is a monad.
\end{proposition}
\begin{proof} Starting with the unit axioms, using both our monad and monoid unit laws, and since the monad multiplication preserves the monoid unit, the following diagrams commute: 
\[ \xymatrixcolsep{5pc}\xymatrixrowsep{3pc}\xymatrix{ \mathsf{S}(A_0) \ar@{}[dr]|-{\text{\normalfont (\ref{diag:monad})}} \ar[r]^-{\mathsf{S}(\eta)} \ar[d]_-{\eta_{\mathsf{S}(A_0)}}& \mathsf{S}\mathsf{S}(A_0) \ar[d]^-{\mu}  \\ 
\mathsf{S}\mathsf{S}(A_0) \ar[r]_-{\mu} & \mathsf{S}(A_0) 
} \]
\[
\xymatrixcolsep{5pc}\xymatrixrowsep{3pc}\xymatrix{ \mathsf{S}(A_0) \otimes A_1 \ar[rr]^-{\mathsf{S}(\eta)\otimes\mathsf{u} \otimes 1} \ar[dd]_-{\mathsf{u} \otimes 1 \otimes 1} \ar@{=}[ddrr]^-{} \ar[ddr]|-{\mathsf{u}\otimes 1 \otimes 1} \ar[drr]|-{1  \otimes \mathsf{u} \otimes 1} & &\mathsf{S}\mathsf{S}(A_0) \otimes \mathsf{S}(A_0) \otimes A_1 \ar[d]^-{\mu \otimes 1 \otimes 1} \ar@{}[dl]|(0.3){\text{\normalfont (\ref{diag:monad})}}  \\ 
 & &  \mathsf{S}A_0 \otimes \mathsf{S}(A_0) \otimes A_1  \ar@{}[dl]|(0.3){\text{\normalfont (\ref{diag:monoid})}}  \ar[d]^-{\mathsf{m}\otimes 1}\\
 \mathsf{S}\mathsf{S}(A_0) \otimes \mathsf{S}(A_0) \otimes A_1 \ar@{}[ur]|(0.5){\text{\normalfont (\ref{diag:monoid-morph})}} \ar[r]_-{\mu \otimes 1 \otimes 1} & \mathsf{S}(A_0) \otimes \mathsf{S}(A_0) \otimes A_1 \ar@{}[u]|(0.5){\text{\normalfont (\ref{diag:monoid})}} \ar[r]_-{\mathsf{m} \otimes 1} & \mathsf{S}(A_0) \otimes A_1.
}
\]
Together, these imply that $\overline{\mu}\circ \overline{\eta}_{\overline{\mathsf{S}}}=1=\overline{\mu}\circ\overline{\mathsf{S}}(\overline{\eta})$. On the other hand, using both our monad and monoid associativity laws, and since the monad multiplication preserves the monoid multiplication, we have that the following diagrams commute: 
\[\xymatrixcolsep{3pc}\xymatrixrowsep{3pc}\xymatrix{
\mathsf{S}\mathsf{S}\mathsf{S}(A_0) \ar[r]^-{\mu} \ar[d]_-{\mathsf{S}(\mu)} \ar@{}[dr]|-{\text{\normalfont (\ref{diag:monad})}} & \mathsf{S}\mathsf{S}(A_0) \ar[d]^-{\mu} \\
\mathsf{S}\mathsf{S}(A_0) \ar[r]_-{\mu} & \mathsf{S}(A_0) }
\]
\[ \xymatrixcolsep{4pc}\xymatrixrowsep{4pc}\xymatrix{
\mathsf{S}\mathsf{S}\mathsf{S}(A_0) \otimes \mathsf{S}\mathsf{S}(A_0) \otimes \mathsf{S}(A_0) \otimes A_1 \ar@{}[dr]|-{\text{\normalfont (\ref{diag:monad})}} \ar[r]^-{\mu \otimes 1 \otimes 1 \otimes 1} \ar[d]_-{\mathsf{S}(\mu) \otimes \mu \otimes 1 \otimes 1} & \mathsf{S}\mathsf{S}(A_0) \otimes\mathsf{S}\mathsf{S}(A_0) \otimes \mathsf{S}(A_0) \otimes A_1 \ar[r]^-{\mathsf{m}\otimes 1 \otimes 1} \ar[d]^-{\mu \otimes \mu \otimes 1 \otimes 1} \ar@{}[dr]|-{\text{\normalfont (\ref{diag:monoid-morph})}} & \mathsf{S}\mathsf{S}(A_0) \otimes \mathsf{S}(A_0) \otimes A_1 \ar[d]^-{\mu \otimes 1 \otimes 1}\\
\mathsf{S}\mathsf{S}(A_0) \otimes \mathsf{S}(A_0) \otimes \mathsf{S}(A_0) \otimes A_1 \ar[d]_-{1 \otimes \mathsf{m} \otimes 1} \ar[r]_-{\mu \otimes 1 \otimes 1 } & \mathsf{S}(A_0) \otimes \mathsf{S}(A_0) \otimes \mathsf{S}(A_0) \otimes A_1  \ar[d]_-{1 \otimes \mathsf{m} \otimes 1} \ar[r]^-{\mathsf{m} \otimes 1 \otimes 1} \ar@{}[dr]|-{\text{\normalfont (\ref{diag:monoid})}} & \mathsf{S}(A_0) \otimes \mathsf{S}(A_0) \otimes A_1 \ar[d]^-{\mathsf{m} \otimes 1} \\
\mathsf{S}\mathsf{S}(A_0) \otimes \mathsf{S}(A_0) \otimes A_1 \ar[r]_-{\mu \otimes 1 \otimes 1 } & \mathsf{S}(A_0) \otimes \mathsf{S}(A_0) \otimes A_1 \ar[r]_-{\mathsf{m} \otimes 1} & \mathsf{S}(A_0) \otimes A_1.}
\]
Together, these imply that $\overline{\mu}\circ\overline{\mathsf{S}}(\overline{\mu})=\overline{\mu}\circ\overline{\mu}_{\overline{\mathsf{S}}}$. Therefore, we conclude that $(\overline{\mathsf{S}}, \overline{\mu}, \overline{\eta})$ is a monad.
\end{proof}

\section{Derivations as Algebras}\label{sec:deri-alg}

In this section, we show that the algebras of the monad we constructed in the previous section are precisely $\mathsf{S}$-derivations. Once again, throughout this section, let $\mathbb{X}$ be a differential category with differential modality $(\mathsf{S}, \mu, \eta, \mathsf{m}, \mathsf{u})$. 

\smallskip

Let us first work out what $\overline{\mathsf{S}}$-algebra would be explicitly. So an $\overline{\mathsf{S}}$-algebra would be a tuple:
\[\left( \left( \phi: A_0 \to A_1 \right), (\nu_0, \nu_1) \right)\]
consisting of a map $\phi: A_0 \to A_1$ in $\mathbb{X}$ and a map $(\nu_0, \nu_1): \overline{\mathsf{S}}\left( \phi: A_0 \to A_1 \right) \to \left( \phi: A_0 \to A_1 \right)$ in the arrow category. So this means we have maps $\nu_0: \mathsf{S}(A_0) \to A_0$ and $\nu_1: \mathsf{S}(A_0) \otimes A_1 \to A_1$ such that the following diagram commutes: 
\begin{equation}\begin{gathered}\label{diag:overlineS-alg-maps} \xymatrixcolsep{5pc}\xymatrix{ \mathsf{S}(A_0) \ar[d]_-{\nu_0} \ar[r]^-{\mathsf{d}_{A_0}}  & \mathsf{S}(A_0) \otimes A_0 \ar[r]^-{1_{\mathsf{S}(A_0)} \otimes \phi} & \mathsf{S}(A_0) \otimes A_1 \ar[d]^-{\nu_1} \\ 
A_0 \ar[rr]_-{\phi} & & A_1.
  } 
\end{gathered}\end{equation}
Now expanding out the $\overline{\mathsf{S}}$-algebra axioms (\ref{diag:S-alg}) gives us the following four diagrams: 
\begin{equation}\begin{gathered}\label{diag:overlineS-alg} \xymatrixcolsep{3pc}\xymatrix{  A_0 \ar[r]^-{\eta_{A_0}} \ar@{=}[dr]^-{}& \mathsf{S}(A_0)  \ar[d]^-{\nu_0} & A_1 \ar[r]^-{\mathsf{u}_{A_0} \otimes 1_{A_1}}\ar@{=}[dr]^-{} & \mathsf{S}(A_0) \otimes A_1 \ar[d]^-{\nu_1} \\
   & A_0 & & A_1 \\ 
\mathsf{S}  \mathsf{S}(A_0) \ar[r]^-{\mathsf{S}(\nu_0)}  \ar[d]_-{\mu_{A_0}} & \mathsf{S}(A_0)  \ar[d]^-{\nu_0} & \mathsf{S}\mathsf{S}(A_0) \otimes \mathsf{S}(A_0) \otimes A_1 \ar[r]^-{\mu_{A_0} \otimes  1_{\mathsf{S}(A_0)} \otimes 1_{A_1}} \ar[d]_-{\mathsf{S}(\nu_0) \otimes \nu_1} & \mathsf{S}(A_0) \otimes \mathsf{S}(A_0) \otimes A_1 \ar[r]^-{\mathsf{m}_{A_0}\otimes 1_{A_1}} & \mathsf{S}(A_0) \otimes A_1 \ar[d]^-{\nu_1} \\
      \mathsf{S}(A_0)  \ar[r]_-{\nu_0} & A_0 & \mathsf{S}(A_0) \otimes A_1 \ar[rr]_-{\nu_1} && A_1  
  } 
\end{gathered}\end{equation}
where the top two diagrams come from $(\nu_0, \nu_1) \circ \overline{\eta}_{\left(\phi: A_0 \to A_1 \right)} = 1_{\left(\phi: A_0 \to A_1 \right)}$, while the bottom two diagrams come from $(\nu_0, \nu_1) \circ \overline{\mu}_{\left(\phi: A_0 \to A_1 \right)} = (\nu_0, \nu_1) \circ \overline{\mathsf{S}}(\nu_0, \nu_1)$. Taking the two left most diagrams immediately tells us that: 

\begin{lemma} If $\left( \left( \phi: A_0 \to A_1 \right), (\nu_0, \nu_1) \right)$ is an $\overline{\mathsf{S}}$-algebra, then $(A_0, \nu_0)$ is an $\mathsf{S}$-algebra. 
\end{lemma}

Since an $\mathsf{S}$-derivation is a map from an $\mathsf{S}$-algebra to a module over such algebra, we would now like to show that $A_1$ is in fact an $(A_0, \nu_0)$-module. So define the map $\alpha_{\nu_1}: A_0 \otimes A_1 \to A_1$ as the following composite: 
\begin{align}
\alpha_{\nu_1} := \xymatrixcolsep{5pc}\xymatrix{ A_0 \otimes A_1 \ar[r]^-{\eta_{A_0} \otimes 1_{A_1}} & \mathsf{S}(A_0) \otimes A_1 \ar[r]^-{\nu_1} & A_1.
  } 
\end{align}
Before we prove that this indeed gives an action, we need the following technical lemma: 

\begin{lemma}If $\left( \left( \phi: A_0 \to A_1 \right), (\nu_0, \nu_1) \right)$ is an $\overline{\mathsf{S}}$-algebra, then the following diagrams commute: 
\begin{equation}\begin{gathered}\label{diag:eta-nu0-nu1} \xymatrixcolsep{5pc}\xymatrix{ \mathsf{S}(A_0) \otimes A_1 \ar[drr]_-{\nu_1}  \ar[r]^-{\nu_0 \otimes 1_{A_1}} & A_0 \otimes A_1 \ar[r]^-{\eta_{A_0} \otimes 1_{A_1}} & \mathsf{S}(A_0) \otimes A_1 \ar[d]^-{\nu_1}  \\
& & A_1
  } 
\end{gathered}\end{equation}
\begin{equation}\begin{gathered}\label{diag:nu1-nu1} \xymatrixcolsep{5pc}\xymatrix{ \mathsf{S}(A_0) \otimes \mathsf{S}(A_0) \otimes A_1 \ar[d]_-{\mathsf{m}_{A_0} \otimes 1_{A_1}}  \ar[r]^-{1_{\mathsf{S}(A_0)} \otimes \nu_1} & \mathsf{S}(A_0) \otimes A_1 \ar[d]^-{\nu_1}  \\
 \mathsf{S}(A_0) \otimes A_1 \ar[r]_-{\nu_1} & A_1.
  } 
\end{gathered}\end{equation}
\end{lemma}

\begin{proof} For the first diagram, using the first and fourth diagrams of (\ref{diag:overlineS-alg}), and the monad and monoid unit axioms, gives us that the following diagram commutes: 
\[ \xymatrixcolsep{5pc}\xymatrix{ \mathsf{S}(A_0) \otimes A_1 \ar@{=}[ddd]^-{}  \ar[dr]_-{\eta \otimes \mathsf{u} \otimes 1} \ar[r]^-{\nu_0 \otimes 1} & A_0 \otimes A_1 \ar@{}[d]|-{\text{\normalfont Nat. $\eta$ + (\ref{diag:overlineS-alg})}} \ar[r]^-{\eta \otimes 1} & \mathsf{S}(A_0) \otimes A_1 \ar[ddd]^-{\nu_1}  \\
& \mathsf{S}\mathsf{S}(A_0) \otimes \mathsf{S}(A_0) \otimes A_1 \ar@{}[dr]|-{\text{\normalfont (\ref{diag:overlineS-alg})}} \ar[d]^-{\mu \otimes  1 \otimes 1}  \ar[ur]_-{\mathsf{S}(\nu_0) \otimes \nu_1}  \ar@{}[dl]|-{\text{\normalfont (\ref{diag:monad}) + (\ref{diag:monoid})}}  \\ 
& \mathsf{S}(A_0) \otimes \mathsf{S}(A_0) \otimes A_1 \ar[dl]_-{\mathsf{m}\otimes 1} &  \\
 \mathsf{S}(A_0) \otimes A_1 \ar[rr]_-{\nu_1} & &  A_1.
  } \]
For the second diagram, we again use the first and fourth diagrams of (\ref{diag:overlineS-alg}) to show that the following diagram commutes: 
  \[ \xymatrixcolsep{5pc}\xymatrix{     \mathsf{S}(A_0) \otimes \mathsf{S}(A_0) \otimes A_1 \ar@{}[drr]|-{\text{\normalfont (\ref{diag:overlineS-alg})}} \ar[rr]^-{1_{\mathsf{S}(A_0)} \otimes \nu_1}  \ar[dr]^-{\mathsf{S}(\eta) \otimes 1 \otimes 1} \ar@{=}[d]_-{} & & \mathsf{S}(A_0) \otimes A_1   \ar[dd]^-{\nu_1} \\ 
  \mathsf{S}(A_0) \otimes \mathsf{S}(A_0) \otimes A_1 \ar@{}[drr]|-{\text{\normalfont (\ref{diag:overlineS-alg})}}  \ar@{}[ur]|(0.3){\text{\normalfont (\ref{diag:monad})}} \ar[d]_-{\mathsf{m} \otimes 1} & \mathsf{S}\mathsf{S}(A_0) \otimes \mathsf{S}(A_0) \otimes A_1 \ar[l]_-{\mu \otimes 1 \otimes 1}    \ar[ur]_-{~~~\mathsf{S}(\nu_0) \otimes \nu_1} & \\
  \mathsf{S}(A_0) \otimes A_1\ar[rr]_-{\nu_1}  & & A_1.
  }  \]  
So the desired diagrams both commute. 
\end{proof}

\begin{lemma} If $\left( \left( \phi: A_0 \to A_1 \right), (\nu_0, \nu_1) \right)$ is an $\overline{\mathsf{S}}$-algebra, then $(A_1, \alpha_{\nu_1})$ is an $(A_0, \nu_0)$-module. 
\end{lemma}
\begin{proof} To prove the unit axiom, we use (\ref{diag:eta-nu0-nu1}) to show that the following diagram commutes: 
\[ \xymatrixcolsep{5pc}\xymatrix{ A_1 \ar@/_2pc/@{=}[ddrr]^-{} \ar[r]^-{\mathsf{u} \otimes 1} & \mathsf{S}(A_0) \otimes A_1 \ar@{}[d]|-{\text{\normalfont (\ref{diag:overlineS-alg})}} \ar[r]^-{\nu_0 \otimes 1}  \ar[ddr]_-{\nu_1} & A_0 \otimes A_1  \ar[d]^-{\eta \otimes 1} \ar@{}[dl]|-{\text{\normalfont ~~~~~~~~~~(\ref{diag:eta-nu0-nu1})}} \\
&  & \mathsf{S}(A_0) \otimes A_1 \ar[d]^-{\nu_1} \\
& & A_1 
  }  \]
which gives us that $\alpha_{\nu_1} \circ (1 \otimes \mathsf{u}^{\nu_0}) = 1$. For the associative axiom, using (\ref{diag:eta-nu0-nu1}) and (\ref{diag:nu1-nu1}), we get that the following diagram commutes: 
\[ \xymatrixcolsep{5pc}\xymatrix{ A_0 \otimes A_0 \otimes A_1  \ar[d]_-{1 \otimes \eta \otimes 1} \ar[r]^-{\eta \otimes \eta \otimes 1} & \mathsf{S}(A_0) \otimes  \mathsf{S}(A_0) \otimes A_1  \ar[dd]_-{1 \otimes \nu_1} \ar[r]^-{\mathsf{m} \otimes 1} &  \mathsf{S}(A_0) \otimes A_1  \ar[ddr]_-{\nu_1} \ar[r]^-{\nu_0 \otimes 1}    & A_0 \otimes A_1   \ar[d]^-{\eta \otimes 1}  \ar@{}[dl]|-{\text{\normalfont ~~~~~~~(\ref{diag:eta-nu0-nu1})}}  \\
A_0 \otimes \mathsf{S}(A_0) \otimes A_1  \ar[d]_-{1 \otimes \nu_1} & & &  \mathsf{S}(A_0) \otimes A_1  \ar[d]^-{\nu_1}\\
A_0 \otimes A_1 \ar[r]_-{\eta \otimes 1} & \mathsf{S}(A_0) \otimes A_1  \ar@{}[uur]|-{\text{\normalfont (\ref{diag:nu1-nu1})}} \ar[rr]_-{\nu_1} & & A_1 
  }  \]
which gives us that $\alpha_{\nu_1} \circ (1 \otimes  \alpha_{\nu_1}) =  \alpha_{\nu_1} \circ (\mathsf{m}^{\nu_0} \otimes 1)$. Therefore we conclude that $(A_1, \alpha_{\nu_1})$ is an $(A_0, \nu_0)$-module. 
\end{proof}

Finally, we show that that the underlying object of our $\overline{\mathsf{S}}$-algebra is in fact an $\mathsf{S}$-derivation. 

\begin{proposition}\label{prop:overlineSalg-to-der} If $\left( \left( \phi: A_0 \to A_1 \right), (\nu_0, \nu_1) \right)$ is an $\overline{\mathsf{S}}$-algebra, then $\phi: (A_0, \nu_0) \to (A_1, \alpha_{\nu_1})$ is an $\mathsf{S}$-derivation.
\end{proposition}
\begin{proof} By (\ref{diag:overlineS-alg-maps}) and (\ref{diag:eta-nu0-nu1}), we have that the following diagram commutes: 
\[   \xymatrixrowsep{1.75pc}\xymatrixcolsep{5pc}\xymatrix{
\mathsf{S}(A_0)\ar[rrr]^-{\nu_0} \ar[dd]_-{\mathsf{d}}  &&& A_0 \ar[dd]^-{\phi} \ar@{}[dll]|-{\text{\normalfont (\ref{diag:overlineS-alg-maps})}}   \\
& \mathsf{S}(A_0) \otimes A_1 \ar[drr]^-{\nu_1} \ar@{}[d]|-{\text{\normalfont (\ref{diag:eta-nu0-nu1})}}  \\ 
 \mathsf{S}(A_0) \otimes A_0 \ar[ur]^-{1 \otimes \phi} \ar[r]_-{\nu_0 \otimes \phi} & A_0 \otimes A_1 \ar[r]_-{\eta \otimes 1} &  \mathsf{S}(A_0) \otimes A_1 \ar[r]_-{\nu_1} & A_1. 
}\]
Therefore $\phi$ satisfies (\ref{diag:derivation}), and so $\phi$ is an $\mathsf{S}$-derivation as desired. 
\end{proof}

We will now define a functor from the category of $\overline{\mathsf{S}}$-algebras to the category of $\mathsf{S}$-derivations. To do so, let first consider an $\overline{\mathsf{S}}$-algebra morphism $(f_0, f_1): \left( \left( \phi: A_0 \to A_1 \right), (\nu_0, \nu_1) \right) \to \left( \left( \phi^\prime: A^\prime_0 \to A^\prime_1 \right), (\nu^\prime_0, \nu^\prime_1) \right)$. Firstly, this is a map in the arrow category $(f_0, f_1): \left( \phi: A_0 \to A_1 \right) \to \left( \phi^\prime: A^\prime_0 \to A^\prime_1 \right)$, so we have maps $f_0: A_0 \to A^\prime_0$ and $f_1: A_1 \to A_1^\prime$ which satisfy (\ref{diag:arrow-morph}). Expanding out (\ref{diag:SAlg-morph}) tells us that the following diagrams commute: 
\begin{equation}\begin{gathered}\label{diag:overlineSAlg-morph}\xymatrixrowsep{1.75pc}\xymatrixcolsep{5pc}\xymatrix{\mathsf{S}(A_0) \ar[r]^-{\mathsf{S}(f_0)}  \ar[d]_-{\nu_0} & \mathsf{S}(A_0^\prime)  \ar[d]^-{\nu_0^\prime} & \mathsf{S}(A_0) \otimes A_1 \ar[r]^-{\mathsf{S}(f_0) \otimes f_1}  \ar[d]_-{\nu_1} & \mathsf{S}(A^\prime_0) \otimes A_1  \ar[d]^-{\nu^\prime_1} \\
      A_0 \ar[r]_-{f_0} & A_0^\prime & A_1 \ar[r]_-{f_1} & A^\prime_1 . }  \end{gathered}\end{equation}
So we define a functor $\mathcal{F}: \overline{\mathsf{S}}\text{-}\mathsf{ALG} \to \mathsf{S}\text{-}\mathsf{DER}$ on objects as \[\mathcal{F}\left( \left( \phi: A_0 \to A_1 \right), (\nu_0, \nu_1) \right) := \left( \phi: (A_0, \nu_0) \to (A_1, \alpha_{\nu_1}) \right),\] and on maps as $\mathcal{F}(f_0,f_1) = (f_0,f_1)$. 

\begin{proposition} $\mathcal{F}: \overline{\mathsf{S}}\text{-}\mathsf{ALG} \to \mathsf{S}\text{-}\mathsf{DER}$ is a functor. 
\end{proposition}
\begin{proof} First we need to explain why $\mathcal{F}$ is well-defined. By Prop \ref{prop:overlineSalg-to-der}, $\mathcal{F}$ is well-defined on objects. On maps, let $(f_0, f_1): \left( \left( \phi: A_0 \to A_1 \right), (\nu_0, \nu_1) \right) \to \left( \left( \phi^\prime: A^\prime_0 \to A^\prime_1 \right), (\nu^\prime_0, \nu^\prime_1) \right)$ be an $\overline{\mathsf{S}}$-algebra morphism. The left diagram of (\ref{diag:overlineSAlg-morph}) tells us that $f_0: (A_0, \nu_0) \to (A^\prime_0, \nu^\prime_0)$ is an $\mathsf{S}$-algebra morphism. So it remains to show that $(f_0, f_1)$ satisfies the two diagrams from (\ref{diag:derivation-morph}). However, notice that $(f_0, f_1)$ satisfying (\ref{diag:arrow-morph}) precisely implies that the right diagram of (\ref{diag:derivation-morph}) holds. On the other hand, using the right diagram of (\ref{diag:overlineSAlg-morph}), we have that the following diagram commutes
\[\xymatrixrowsep{1.75pc}\xymatrixcolsep{5pc}\xymatrix{A_0 \otimes A_1 \ar@{}[dr]|-{\text{\normalfont Nat. of $\eta$}} \ar[d]_-{f_0 \otimes f_1} \ar[r]^-{\eta \otimes 1} & \mathsf{S}(A_0) \otimes A_1 \ar@{}[dr]|-{\text{\normalfont (\ref{diag:overlineSAlg-morph})}} \ar[d]^-{\mathsf{S}(f_0) \otimes f_1} \ar[r]^-{\nu_1} & A_1   \ar[d]^-{f_1} \\
A^\prime_0 \otimes A^\prime_1 \ar[r]_-{\eta \otimes 1} & \mathsf{S}(A^\prime_0) \otimes A^\prime_1 \ar[r]_-{\nu^\prime_1} & A^\prime_1} \]
which gives us that $(f_0, f_1)$ also satisfies the left diagram of (\ref{diag:derivation-morph}). Thus we have that 
\[(f_0,f_1): \left( \phi: (A_0, \nu_0) \to (A_1, \alpha_{\nu_1}) \right) \to \left( \phi^\prime: (A^\prime_0, \nu^\prime_0) \to (A^\prime_1, \alpha_{\nu^\prime_1}) \right)\]
is indeed a map in $\mathsf{S}\text{-}\mathsf{DER}$. So $\mathcal{F}$ is well-defined. Functoriality of $\mathcal{F}$ is essentially automatic. 
\end{proof}

Let us now go the other direction. So suppose $\mathsf{D}: (A,\nu) \to (M,\alpha)$ is an $\mathsf{S}$-derivation. Of course this means that $\mathsf{D}: A \to M$ is an object in the arrow category, and expanding out $\overline{\mathsf{S}}\left( \mathsf{D}: A \to M \right)$ gives us the composite: 
\begin{align*}
\overline{\mathsf{S}}\left(\mathsf{D}: A \to M \right):= \left( \xymatrixcolsep{3pc}\xymatrix{  \mathsf{S}(A) \ar[r]^-{\mathsf{d}_{A}} & \mathsf{S}(A) \otimes A \ar[r]^-{1_{\mathsf{S}(A)} \otimes \mathsf{D}} & \mathsf{S}(A) \otimes M 
  }\right).
\end{align*}
Now define the map $\nu_{\alpha}: \mathsf{S}(A) \otimes M \to M$ as the following composite: 
\begin{align}
\nu_{\alpha} := \xymatrixcolsep{5pc}\xymatrix{ \mathsf{S}(A) \otimes M \ar[r]^-{\nu \otimes 1_{M}} & A \otimes M \ar[r]^-{\alpha} & M.
  } 
\end{align}

\begin{lemma} $(\nu, \nu_\alpha): \overline{\mathsf{S}}\left( \mathsf{D}: A \to M \right) \to \left( \mathsf{D}: A \to M \right)$ is a map in the arrow category. 
\end{lemma}
\begin{proof} We may rewrite the $\mathsf{S}$-derivation chain rule as follows: 
\[ \xymatrixrowsep{1.75pc}\xymatrixcolsep{5pc}\xymatrix{  \mathsf{S}(A) \textbf{\ar@{}[ddrr]|-{\text{\normalfont (\ref{diag:derivation})}}} \ar[dd]_-{\nu}  \ar[r]^-{\mathsf{d}} & \mathsf{S}(A) \otimes A \ar[r]^-{1 \otimes \mathsf{D}} & \mathsf{S}(A) \otimes M \ar[d]^-{\nu \otimes 1} \\
& & A \otimes M \ar[d]^-{\alpha} \\
A \ar[rr]_-{\mathsf{D}} & & M 
  } \]
which says precisely that $(\nu, \nu_\alpha): \overline{\mathsf{S}}\left( \mathsf{D}: A \to M \right) \to \left( \mathsf{D}: A \to M \right)$ is a map in the arrow category. 
\end{proof}

\begin{proposition}\label{prop:der-to-overlineSalg} If $\mathsf{D}: (A,\nu) \to (M,\alpha)$ is an $\mathsf{S}$-derivation, then $\left(\left( \mathsf{D}: A \to M \right), (\nu, \nu_{\alpha}) \right)$ is an $\overline{\mathsf{S}}$-algebra. 
\end{proposition}
\begin{proof} We need to show that the four diagrams of (\ref{diag:overlineS-alg}) hold. However notice that since $(A,\nu)$ is an $\mathsf{S}$-algebra and $(M, \alpha)$ is an $(A,\nu)$-module, the two left most diagrams and the top right diagram of (\ref{diag:overlineS-alg}) already hold. So we need only show the bottom right diagram of (\ref{diag:overlineS-alg}). To do so, recall that every $\mathsf{S}$-algebra morphism is also a monoid morphism. So since $\nu: (\mathsf{S}(A), \mu_A) \to (A,\nu)$ is an $\mathsf{S}$-algebra morphism, we also have that $\nu: (\mathsf{S}(A), \mathsf{m}_A, \mathsf{u}_A) \to (A, \mathsf{m}^\nu, \mathsf{u}^\nu)$ is a monoid morphism. Using this, and the $\mathsf{S}$-algebra and module associativity axioms, we have that the following diagram commutes: 
\[ \xymatrixrowsep{1.75pc}\xymatrixcolsep{5pc}\xymatrix{  \mathsf{S}\mathsf{S}(A) \otimes \mathsf{S}(A) \otimes M \textbf{\ar@{}[dr]|-{\text{\normalfont (\ref{diag:S-alg})}}}  \ar[d]_-{\mathsf{S}(\nu) \otimes \nu \otimes 1} \ar[r]^-{\mu \otimes 1 \otimes 1} & \mathsf{S}(A) \otimes \mathsf{S}(A) \otimes M \ar[d]_-{\nu \otimes \nu \otimes 1} \ar[r]^-{\mathsf{m} \otimes 1} \textbf{\ar@{}[dr]|-{\text{\normalfont (\ref{diag:monoid-morph})}}} &  \mathsf{S}(A) \otimes M \ar[d]^-{\nu \otimes 1} \\
\mathsf{S}(A) \otimes A \otimes M  \ar[r]_-{\nu \otimes 1 \otimes 1} \ar[d]_-{1 \otimes \alpha} & A \otimes A \otimes M \textbf{\ar@{}[dr]|-{\text{\normalfont (\ref{diag:module})}}}  \ar[d]_-{1 \otimes \alpha} \ar[r]^-{\mathsf{m}^\nu \otimes 1} & A \otimes M \ar[d]^-{1 \otimes \alpha} \\ 
\mathsf{S}(A) \otimes M \ar[r]_-{\nu \otimes 1} & A \otimes M \ar[r]_-{\alpha} & M 
  } \]
giving us that the bottom right diagram of (\ref{diag:overlineS-alg}) holds. Therefore, we have that $(\nu, \nu_{\alpha}): \overline{\mathsf{S}}\left( \mathsf{D}: A \to M \right) \to \left( \mathsf{D}: A \to M \right)$ is a map in the arrow category. 
\end{proof}

We now define a functor $\mathcal{F}^{-1}: \mathsf{S}\text{-}\mathsf{DER} \to \overline{\mathsf{S}}\text{-}\mathsf{ALG}$ on objects as \[\mathcal{F}^{-1}\left( \mathsf{D}: (A,\nu) \to (M,\alpha)\right) = \left(\left( \mathsf{D}: A \to M \right), (\nu, \nu_{\alpha}) \right)\] and on maps as $\mathcal{F}^{-1}(f,g) = (f,g)$. 

\begin{proposition} $\mathcal{F}^{-1}: \mathsf{S}\text{-}\mathsf{DER} \to \overline{\mathsf{S}}\text{-}\mathsf{ALG}$ is a functor. 
\end{proposition}
\begin{proof} First we need to explain why $\mathcal{F}^{-1}$ is well-defined. By Prop \ref{prop:der-to-overlineSalg}, $\mathcal{F}$ is well-defined on objects. On maps, let $(f, g): \left(\mathsf{D}: (A,\nu) \to (M,\alpha) \right) \to \left( \mathsf{D}^\prime: (A^\prime,\nu^\prime) \to (M^\prime,\alpha^\prime) \right)$ be a map in $\mathsf{S}\text{-}\mathsf{DER}$, which recall means that $f: (A,\nu) \to (A^\prime,\nu^\prime)$ is an $\mathsf{S}$-algebra morphism and (\ref{diag:derivation-morph}) holds. We first need to explain why $(f,g)$ is a map in the arrow category. However note that the right diagram of (\ref{diag:derivation-morph}) precisely says that $(f,g)$ satisfies (\ref{diag:arrow-morph}), and so $(f,g): \left( \mathsf{D}: A \to M \right) \to \left( \mathsf{D}^\prime: A^\prime \to M^\prime \right)$ is a map in the arrow category. Next we need to show that $(f,g)$ also satisfies the two diagram of (\ref{diag:overlineSAlg-morph}). However notice that since $f$ is an $\mathsf{S}$-algebra morphism, (\ref{diag:SAlg-morph}) is the same as the left diagram of (\ref{diag:overlineSAlg-morph}). So it remains to show the right diagram of (\ref{diag:overlineSAlg-morph}). To do so, combining that $f$ is an $\mathsf{S}$-algebra morphism with the left diagram of (\ref{diag:derivation-morph}) gives us that the following diagram commutes: 
\[ \xymatrixrowsep{1.75pc}\xymatrixcolsep{5pc}\xymatrix{ \mathsf{S}(A) \otimes M \textbf{\ar@{}[dr]|-{\text{\normalfont (\ref{diag:SAlg-morph})}}}  \ar[d]_-{\nu \otimes 1} \ar[r]^-{\mathsf{S}(f) \otimes g} & \mathsf{S}(A^\prime) \otimes M^\prime \ar[d]^-{\nu^\prime \otimes 1} \\
A \otimes M \ar[d]_-{\alpha}  \textbf{\ar@{}[dr]|-{\text{\normalfont (\ref{diag:derivation-morph})}}} \ar[r]^-{f \otimes g}& A^\prime \otimes M^\prime \ar[d]^-{\alpha^\prime} \\ 
M \ar[r]_-{g} &  M^\prime. 
  } \]
So $(f,g): \left(\left( \mathsf{D}: A \to M \right), (\nu, \nu_{\alpha}) \right) \to \left(\left( \mathsf{D}^\prime: A^\prime \to M^\prime \right), (\nu^\prime, \nu^\prime_{\alpha^\prime}) \right)$ is indeed an $\overline{\mathsf{S}}$-algebra morphism. So $\mathcal{F}^{-1}$ is well-defined. Functoriality of $\mathcal{F}^{-1}$ is essentially automatic, so we get that $\mathcal{F}^{-1}$ is indeed a functor. 
\end{proof}

Bringing all this together, we conclude this section with the main result of this paper. 

\begin{theorem}\label{thm:deri=alg} We have an isomorphism of category $\overline{\mathsf{S}}\text{-}\mathsf{ALG} \cong \mathsf{S}\text{-}\mathsf{DER}$. Explicitly, the functors from Prop \ref{prop:overlineSalg-to-der} and Prop \ref{prop:der-to-overlineSalg} are inverses of each other. 
\end{theorem}
\begin{proof} It suffices to show that the constructions of Prop \ref{prop:overlineSalg-to-der} and Prop \ref{prop:der-to-overlineSalg} are inverses of each other. So let us start with an $\overline{\mathsf{S}}$-algebra $\left( \left( \phi: A_0 \to A_1 \right), (\nu_0, \nu_1) \right)$ and build its induced $\mathsf{S}$-derivation $\phi: (A_0, \nu_0) \to (A_1, \alpha_{\nu_1})$. Applying Prop \ref{prop:der-to-overlineSalg} to this $\mathsf{S}$-derivation gives us the $\overline{\mathsf{S}}$-algebra $\left( \left( \phi: A_0 \to A_1 \right), (\nu_0, \nu_{\alpha_{\nu_1}}) \right)$. To show that this is the same $\overline{\mathsf{S}}$-algebra as the one we started with, we need to show that $\nu_{\alpha_{\nu_1}}$ is equal to $\nu_1$. Expanding out $\nu_{\alpha_{\nu_1}}$ gives us the following composite: 
\[ \xymatrixcolsep{5pc}\xymatrix{ \mathsf{S}(A_0) \otimes A_1 \ar[r]^-{\nu_0 \otimes 1_{A_1}} & A_0 \otimes A_1 \ar[r]^-{\eta_{A_0} \otimes 1_{A_1}} &  \mathsf{S}(A_0) \otimes A_1 \ar[r]^-{\nu_1} & A_1.
  }  \]
  However by (\ref{diag:eta-nu0-nu1}), the above composite is equal to $\nu_1$, so $\nu_{\alpha_{\nu_1}}=\nu_1$. Conversely, suppose we instead start with an $\mathsf{S}$-derivation $\mathsf{D}: (A,\nu) \to (M,\alpha)$ and then build its induced $\overline{\mathsf{S}}$-algebra $\left(\left( \mathsf{D}: A \to M \right), (\nu, \nu_{\alpha}) \right)$. Applying Prop \ref{prop:overlineSalg-to-der} to it gives us the $\mathsf{S}$-derivation $\mathsf{D}: (A,\nu) \to (M,\alpha_{\nu_{\alpha}})$. To show that this is the same $\mathsf{S}$-derivation as the one we started with, we need to show that $\alpha_{\nu_{\alpha}}$ is equal to $\alpha$. Expanding out $\alpha_{\nu_{\alpha}}$ gives us: 
\[ \xymatrixcolsep{5pc}\xymatrix{ A \otimes M \ar[r]^-{\eta_A \otimes 1_{M}} & \mathsf{S}(A) \otimes M \ar[r]^-{\nu \otimes 1_M} &  A \otimes M \ar[r]^-{\alpha} & M.
  }  \]
  However, by (\ref{diag:S-alg}), the above composite is simply equal to $\alpha$, so $\alpha_{\nu_{\alpha}} = \alpha$. From this it easily follows that $\mathcal{F}: \overline{\mathsf{S}}\text{-}\mathsf{ALG} \to \mathsf{S}\text{-}\mathsf{DER}$ and $\mathcal{F}^{-1}: \mathsf{S}\text{-}\mathsf{DER} \to \overline{\mathsf{S}}\text{-}\mathsf{ALG}$ are inverses of each other. So $\overline{\mathsf{S}}\text{-}\mathsf{ALG} \cong \mathsf{S}\text{-}\mathsf{DER}$ as desired. 
\end{proof}

\section{Derivations as Commutative Monoids}\label{sec:der-monoid}

In Sec \ref{sec:monad} we have shown that we can lift our differential modality $\mathsf{S}$ to a monad $\overline{\mathsf{S}}$ on the arrow category. It is then natural to ask whether $\overline{\mathsf{S}}$ is in fact a differential modality itself. We will answer this question in Sec \ref{sec:diff-mod}. To be able to do so, we first need to address what monoidal structure to give the arrow category and then, in particular, to understand what commutative monoids are: this is the topic of this section. 
\smallskip

Naturally, a first choice for a monoidal product on the arrow category would be to simply take the pointwise monoidal product: 
\[\left( \phi: A_0 \to A_1 \right) \otimes (\psi: B_0 \otimes B_1) = \left( \phi \otimes \psi: A_0 \otimes B_0 \to A_1 \otimes B_1 \right).\] 
However this will not work since $\overline{\mathsf{S}}\left( \phi: A_0 \to A_1 \right)$ will not be a monoid for this monoidal product. Instead, we must consider another monoidal structure for the arrow category, one that requires our base category to have \textit{biproducts}. 

\smallskip 

For a category $\mathbb{X}$ with finite biproducts, we write the biproduct of a finite family of objects as $A_1 \oplus \dots \oplus A_n$ and $\mathsf{0}$ for the zero object. For simplicity, we work with strict biproduct, so associativity is strict and $A \oplus \mathsf{0} = A = \mathsf{0} \oplus A$. Now recall that a category with finite biproducts is canonically (and uniquely) enriched over commutative monoids, and also that maps have a very practical matrix representation. Indeed, a map $A_1 \oplus \dots \oplus A_n \to B_1 \oplus \dots \oplus B_m$ can be represented by an $m \times n$ matrix as follows: 
\[ \begin{bmatrix} f_{1,1} & f_{1,2} & \hdots & f_{1,n} \\
f_{2,1} & \ddots & \ddots & f_{2,n} \\
\vdots & \ddots & \ddots & \vdots \\
f_{m,1} & f_{m,2} & \hdots & f_{m,n}
\end{bmatrix}\]
for unique maps $f_{i,j}: A_j \to B_i$. Moreover, composition is given by matrix multiplication. 

\smallskip

Let $\mathbb{X}$ be an additive symmetric monoidal category with finite biproducts. Since the monoidal structure $\otimes$ preserves the additive structure, it follows that the monoidal structure also preserves finite products, that is, we have canonical isomorphisms:
\begin{align}
\left( A_1 \oplus \dots \oplus A_n \right) \otimes \left( B_1 \oplus \dots \oplus B_m \right) \cong \bigoplus\limits_{\substack{1 \leq i \leq n \\ 1 \leq j \leq m}} A_i \otimes B_j && A \otimes \mathsf{0} \cong \mathsf{0} \cong \mathsf{0} \otimes A.
\end{align}
In other words, $\mathbb{X}$ is a distributive monoidal category. 

\smallskip

We can now define a monoidal structure on the arrow category. This monoidal structure will be a special case of (the dual of) the \textbf{Leibniz construction} for bifunctors \cite[Def 4.4]{riehl2013theory}, specifically (the dual of) the \textbf{pushout product} over the zero object \cite[Const 11.1.7]{riehl2014categorical}, which is a construction made to give a monoidal product on the arrow category. 

\begin{definition} We define the  monoidal product $\boxtimes$ on $\mathsf{Arr}[\mathbb{X}]$ on objects as follows: 
\begin{align}
\left(\phi: A_0 \to A_1 \right) \boxtimes \left( \psi: B_0 \to B_1 \right) := \xymatrixcolsep{5pc}\xymatrix{ A_0 \otimes B_0 \ar[r]^-{\begin{bmatrix} 1_{A_0} \otimes \psi \\ \phi \otimes 1_{B_0} \end{bmatrix}} & (A_0 \otimes B_1) \oplus (A_1 \otimes B_0)
  } 
\end{align}
and given two maps $(f_0, f_1): \left( \phi: A_0 \to A_1 \right) \to \left(\phi^\prime: A_0^\prime \to A_1^\prime\right)$ and $(g_0, g_1): \left( \psi: B_0 \to B_1 \right) \to \left(\psi^\prime: B_0^\prime \to B_1^\prime\right)$, their monoidal product \[(f_0, f_1) \boxtimes  (g_0, g_1): \left( \phi: A_0 \to A_1 \right) \boxtimes \left( \psi: B_0 \to B_1 \right) \to \left(\phi^\prime: A_0^\prime \to A_1^\prime\right) \boxtimes \left(\psi^\prime: B_0^\prime \to B_1^\prime\right)\] is defined as the pair of following maps: 
\begin{equation}\begin{gathered}\label{boxtimes-maps} 
\xymatrixcolsep{3pc}\xymatrix{ A_0 \otimes B_0 \ar[r]^-{f_0 \otimes g_0} & A^\prime_0 \otimes B^\prime_0 
  } \\
  \xymatrixcolsep{5pc}\xymatrix{ (A_0 \otimes B_1) \oplus (A_1 \otimes B_0) \ar[rr]^-{\begin{bmatrix} f_0 \otimes g_1 & 0 \\ 0 & f_1 \otimes g_0  \end{bmatrix}} && (A^\prime_0 \otimes B^\prime_1) \oplus (A^\prime_1 \otimes B^\prime_0).
  }  
\end{gathered}\end{equation}
The monoidal unit is the zero map from the monoidal unit to the zero object $0: I \to \mathsf{0}$. The symmetry isomorphism \[\sigma^\boxtimes_{\left(\phi: A_0 \to A_1 \right), \left( \psi: B_0 \to B_1 \right)}: \left(\phi: A_0 \to A_1 \right) \boxtimes \left( \psi: B_0 \to B_1 \right) \to \left( \psi: B_0 \to B_1 \right)  \boxtimes \left(\phi: A_0 \to A_1 \right)\] is defined as the pair of following morphisms: 
\begin{equation}\begin{gathered}\label{symmetry-box} 
\xymatrixcolsep{3pc}\xymatrix{ A_0 \otimes B_0 \ar[r]^-{\sigma_{A_0, B_0}} & B_0 \otimes A_0  
  } \\
  \xymatrixcolsep{6pc}\xymatrix{ (A_0 \otimes B_1) \oplus (A_1 \otimes B_0) \ar[rr]^-{\begin{bmatrix} 0 & \sigma_{A_0,B_1} \\ \sigma_{A_1, B_0}  & 0 \end{bmatrix}} && (B_0 \otimes A_1) \oplus (B_1 \otimes A_0).
  } 
\end{gathered}\end{equation}
\end{definition}

The monoidal product $\boxtimes$ makes $\mathsf{Arr}[\mathbb{X}]$ into a symmetric monoidal category. As mentioned above, this monoidal product $\boxtimes$ is in fact an example of (the dual of) the Leibniz construction and a (dual of the) pushout product -- see \cite{riehl2013theory,riehl2014categorical} for more details. Moreover, $\mathsf{Arr}[\mathbb{X}]$ is also additive where the additive structure is given pointwise, that is, the sum of maps is $(f_0, f_1) + (g_0, g_1) = (f_0 + g_0, f_1 + g_1)$ and the zero maps are $(0,0)$. Furthermore, $\mathsf{Arr}[\mathbb{X}]$ also inherits biproducts from $\mathbb{X}$ which are also given in pointwise fashion. 

\begin{lemma}\label{lem:arrow-monoidal} $\mathsf{Arr}[\mathbb{X}]$ is an additive symmetric monoidal category with monoidal product $\boxtimes$ and monoidal unit ${0: I \to \mathsf{0}}$. Moreover, $\mathsf{Arr}[\mathbb{X}]$ also has finite biproducts. 
\end{lemma}
\begin{proof} It is straightforward (though tedious) to show that $\boxtimes$ indeed gives an additive symmetric monoidal structure, as the necessary coherences follow from distributivity of $\otimes$ and $\oplus$. Alternatively, this follows from the fact that $\boxtimes$ is a (dual of the) Leibniz construction/pushout product \cite{riehl2013theory,riehl2014categorical}. 
\end{proof}

As anticipated, in order to show that $\overline{\mathsf{S}}$ is an algebra modality, it will be useful to work out what the commutative monoids are in the arrow category with respect to $\boxtimes$. It turns out that in this case the commutative monoids correspond precisely to derivations in the usual sense. 

A commutative monoid in $\mathsf{Arr}[\mathbb{X}]$ is a triple
\[\left( \left( \phi: A_0 \to A_1 \right), \left(\mathsf{m}_0, \begin{bmatrix} \mathsf{m}_1 & \mathsf{m}_2 \end{bmatrix}\right), (\mathsf{u}_0, \mathsf{u}_1) \right)\]
consisting of a map $\phi: A_0 \to A_1$ in $\mathbb{X}$ and maps $ \left(\mathsf{m}_0, \begin{bmatrix} \mathsf{m}_1 & \mathsf{m}_2 \end{bmatrix}\right): \left( \phi: A_0 \to A_1 \right) \boxtimes \left( \phi: A_0 \to A_1 \right) \to \left( \phi: A_0 \to A_1 \right)$ and $(\mathsf{u}_0, \mathsf{u}_1): \left( 0: I \to \mathsf{0} \right) \to \left( \phi: A_0 \to A_1 \right)$ in the arrow category, which also satisfy (\ref{diag:monoid}). Expanding out these arrow category maps tells us that we have five maps in $\mathbb{X}$ of the following type: 
\begin{gather*}
\mathsf{m}_0: A_0 \otimes A_0 \to A_0 \qquad \mathsf{m}_1: A_0 \otimes A_1 \to A_1 \qquad \mathsf{m}_2: A_1 \otimes A_0 \to A_1 \\
\mathsf{u}_0: I \to A_0 \qquad \mathsf{u}_1: \mathsf{0} \to A_1
\end{gather*}
where note that $\begin{bmatrix} \mathsf{m}_1 & \mathsf{m}_2 \end{bmatrix}: (A_0 \otimes A_1) \oplus (A_1 \otimes A_0) \to A_1$. Observe also that, since $\mathsf{0}$ is a zero object, we must have $\mathsf{u}_1 = 0$. 
\smallskip

We will now explain how $A_0$ is a commutative monoid and how $A_1$ is a module over it. To do so, let us expand out what (\ref{diag:monoid}) is in the arrow category, which would give us the following six diagrams: 
\begin{equation}\begin{gathered}\label{diag:monoid-assoc-arrow-1}\xymatrixrowsep{1.75pc}\xymatrixcolsep{3pc}\xymatrix{A_0  \otimes  A_0  \otimes  A_0  \ar[r]^-{1_{A_0} \otimes \mathsf{m}_0} \ar[d]_-{\mathsf{m}_0 \otimes 1_{A_0}}  & A_0  \otimes  A_0  \ar[d]^-{\mathsf{m}_0} &   \\
      A_0 \otimes  A_0  \ar[r]_-{\mathsf{m}_0} &A_0} \end{gathered}\end{equation}
      
    \begin{equation}\begin{gathered}\label{diag:monoid-assoc-arrow-2}  \xymatrixrowsep{2pc}\xymatrixcolsep{5.75pc}\xymatrix{ {\begin{matrix} (A_0  \otimes  A_0  \otimes  A_1) \\\oplus   (A_0  \otimes  A_1  \otimes  A_0) \\ \oplus  (A_1  \otimes  A_0  \otimes  A_0)\end{matrix}} \ar[rr]^-{\begin{bmatrix} 1_{A_0}  \otimes \mathsf{m}_1 & 1_{A_0}  \otimes \mathsf{m}_2 & 0 \\
    0 & 0 & 1_{A_1}\otimes  \mathsf{m}_0  \end{bmatrix}} \ar[dd]|-{\begin{bmatrix} \mathsf{m}_0 \otimes 1_{A_1} & 0 & 0 \\ 0 & \mathsf{m}_1 \otimes 1_{A_0} & \mathsf{m}_2 \otimes 1_{A_0}  \end{bmatrix}} & & (A_0 \otimes A_1) \! \oplus \! (A_1 \otimes A_0)  \ar[dd]|-{\begin{bmatrix} \mathsf{m}_1 & \mathsf{m}_2 \end{bmatrix}}  \\ \\
       (A_0 \otimes A_1) \oplus (A_1 \otimes A_0)  \ar[rr]_-{\begin{bmatrix} \mathsf{m}_1 & \mathsf{m}_2 \end{bmatrix}} && A_1} \end{gathered}\end{equation}
       
     \begin{equation}\begin{gathered}\label{diag:monoid-unit-arrow-1}  \xymatrixrowsep{1.75pc}\xymatrixcolsep{5pc}\xymatrix{A_0 \ar@{=}[dr] \ar[r]^-{1_{A_0} \otimes \mathsf{u}_0}\ar[d]_-{\mathsf{u}_0 \otimes 1_{A_0}} & A_0  \otimes  A_0 \ar[d]^-{\mathsf{m}_0}   \\
     A_0 \otimes  A_0 \ar[r]_-{\mathsf{m}_0}   & A_0  } \end{gathered}\end{equation}
     
    \begin{equation}\begin{gathered}\label{diag:monoid-unit-arrow-2} \xymatrixrowsep{1.75pc}\xymatrixcolsep{5pc}\xymatrix{ A_1 \ar@{=}[dr] \ar[r]^-{\begin{bmatrix} 0 & 1_{A_1} \otimes \mathsf{u}_0 \end{bmatrix}} \ar[d]_-{\begin{bmatrix} \mathsf{u}_0 \otimes 1_{A_1} & 0 \end{bmatrix}} & (A_0 \otimes A_1) \oplus (A_1 \otimes A_0) \ar[d]^-{\begin{bmatrix} \mathsf{m}_1 & \mathsf{m}_2 \end{bmatrix}}  \\
     (A_0 \otimes A_1) \oplus (A_1 \otimes A_0) \ar[r]_-{\begin{bmatrix} \mathsf{m}_1 & \mathsf{m}_2 \end{bmatrix}} & A_1  } \end{gathered}\end{equation}
     
    \begin{equation}\begin{gathered}\label{diag:monoid-com-arrow-1} \xymatrixrowsep{1.75pc}\xymatrixcolsep{3pc}\xymatrix{ A_0 \otimes A_0  \ar[r]^-{\sigma_{A_0,A_0}}  \ar[dr]_-{\mathsf{m}_0} & A_0  \otimes  A_0 \ar[d]^-{\mathsf{m}_0} \\
     & A_0  } \end{gathered}\end{equation}
     
    \begin{equation}\begin{gathered}\label{diag:monoid-com-arrow-2} \xymatrixrowsep{1.75pc}\xymatrixcolsep{3pc}\xymatrix{ (A_0 \otimes A_1) \oplus (A_1 \otimes A_0) \ar[drr]_-{\begin{bmatrix} \mathsf{m}_1 & \mathsf{m}_2 \end{bmatrix}~~~} \ar[rr]^-{\begin{bmatrix} 0 & \sigma_{A_0,A_1} \\ \sigma_{A_1, A_0}  & 0 \end{bmatrix}} && (A_0 \otimes A_1) \oplus (A_1 \otimes A_0) \ar[d]^-{\begin{bmatrix} \mathsf{m}_1 & \mathsf{m}_2 \end{bmatrix}} \\
     && A_1. }
       \end{gathered}\end{equation}
 Now (\ref{diag:monoid-assoc-arrow-1}), (\ref{diag:monoid-unit-arrow-1}), and (\ref{diag:monoid-com-arrow-1}) together tell us that $(A_0, \mathsf{m}_0, \mathsf{u}_0)$ is a commutative monoid. Moreover, by rewriting the bottom triangle of (\ref{diag:monoid-unit-arrow-2}) and precomposing (\ref{diag:monoid-assoc-arrow-2}) with the first injection gives us that the two following diagrams also commute: 
\begin{equation}\begin{gathered}\label{}\xymatrixrowsep{1.75pc}\xymatrixcolsep{5pc}\xymatrix{A_0  \otimes  A_0  \otimes  A_1  \ar[r]^-{1_{A_0} \otimes \mathsf{m}_1} \ar[d]_-{\mathsf{m}_0 \otimes 1_{A_1}}  & A_0  \otimes  A_1 \ar[d]^-{\mathsf{m}_1} & A_1 \ar@{=}[dr] \ar[r]^-{\mathsf{u}_0 \otimes 1_{A_1}} & A_0  \otimes  A_1 \ar[d]^-{\mathsf{m}_1}  \\
      A_0 \otimes  A_1 \ar[r]_-{\mathsf{m}_1} & A_1  &  & A_1} \end{gathered}\end{equation} 
 which tells us that $(A_1, \mathsf{m}_1)$ is an $(A_0, \mathsf{m}_0, \mathsf{u}_0)$-module. One may then ask what can be said about the remaining data involving $\mathsf{m}_2$. As it turns out, this information is in fact redundant. Indeed, note that precomposing (\ref{diag:monoid-com-arrow-2}) with the second injection gives us that the following diagram commutes: 
  \begin{equation}\begin{gathered}\label{} \xymatrixrowsep{1.75pc}\xymatrixcolsep{3pc}\xymatrix{ A_1 \otimes A_0  \ar[r]^-{\sigma_{A_1,A_0}}  \ar[dr]_-{\mathsf{m}_2} & A_0  \otimes  A_1 \ar[d]^-{\mathsf{m}_1} \\
     & A_0.  } \end{gathered}\end{equation}
     
Therefore $\mathsf{m}_2$ is completely determined by $\mathsf{m}_1$, in other words, $\mathsf{m}_2$ is the induced right action from the left action $\mathsf{m}_1$. Thus the remaining data from (\ref{diag:monoid-assoc-arrow-1}) and (\ref{diag:monoid-unit-arrow-1}) corresponds to the right action analogues of (\ref{diag:module}) and how the left and right action interact in an obvious way. 

Finally, expanding out that $ \left(\mathsf{m}_0, \begin{bmatrix} \mathsf{m}_1 & \mathsf{m}_1 \circ \sigma_{A_1,A_0} \end{bmatrix}\right): \left( \phi: A_0 \to A_1 \right) \boxtimes \left( \phi: A_0 \to A_1 \right) \to \left( \phi: A_0 \to A_1 \right)$ and $(\mathsf{u}_0, 0): \left( 0: I \to \mathsf{0} \right) \to \left( \phi: A_0 \to A_1 \right)$ are maps in the arrow category gives us the following diagrams: 
\begin{equation}\begin{gathered}\label{}\xymatrixrowsep{1.75pc}\xymatrixcolsep{5pc}\xymatrix{A_0  \otimes  A_0   \ar[r]^-{\begin{bmatrix} \phi \otimes 1_{A_0} \\ 1_{A_0} \otimes \phi \end{bmatrix}} \ar[d]_-{\mathsf{m}_0}  & (A_0  \otimes  A_1) \oplus (A_1 \otimes A_0) \ar[d]^-{\begin{bmatrix} \mathsf{m}_1 & \mathsf{m}_1 \circ \sigma_{A_1,A_0}\end{bmatrix}} & I  \ar[d]_-{\mathsf{u}_0} \ar[r]^-{0} & \mathsf{0} \ar[d]^-{0}  \\
      A_0 \ar[r]_-{\phi} & A_1  &  A_0 \ar[r]_-{\phi}  & A_1} \end{gathered}\end{equation} 
    which we can rewrite as follows: 
    \begin{equation}\begin{gathered}\label{}\xymatrixrowsep{1.75pc}\xymatrixcolsep{7pc}\xymatrix{A_0  \otimes  A_0   \ar[r]^-{ \sigma_{A_1,A_0} \circ (\phi \otimes 1_{A_0}) + 1_{A_0} \otimes \phi} \ar[d]_-{\mathsf{m}_0}  & A_0 \otimes A_1 \ar[d]^-{\mathsf{m}_1} & I \ar[d]_-{\mathsf{u}_0}  \ar[dr]^-{0} &   \\
      A_0 \ar[r]_-{\phi} & A_1  &  A_0 \ar[r]_-{\phi}  & A_1.} \end{gathered}\end{equation} 
      These last diagrams tell us precisely that $\phi: (A_0, \mathsf{m}_0, \mathsf{u}_0) \to (A_1, \mathsf{m}_1)$ is a derivation in $\mathbb{X}$. 

      \smallskip 
      
      Conversely, if $\mathsf{D}: (A, \mathsf{m}, \mathsf{u}) \to (M, \alpha)$ is a derivation in $\mathbb{X}$, then by essentially the reverse argument as above, we get that $\left( \left( \mathsf{D}: A \to M \right), \left(\mathsf{m}, \begin{bmatrix} \alpha & \alpha \circ \sigma_{M, A} \end{bmatrix}\right), (\mathsf{u}, 0) \right)$ is a commutative monoid in $\mathsf{Arr}[\mathbb{X}]$. Moreover, clearly these constructions are inverses of each other. Therefore, we conclude that commutative monoids in the arrow category are preicsely of the form $\left( \left( \mathsf{D}: A \to M \right), \left(\mathsf{m}, \begin{bmatrix} \alpha & \alpha \circ \sigma_{M, A} \end{bmatrix}\right), (\mathsf{u}, 0) \right)$ for some derivation $\mathsf{D}: (A, \mathsf{m}, \mathsf{u}) \to (M, \alpha)$ in the base category. Furthermore, it is easy to see that this extends to maps and that monoid morphisms in the arrow category are precisely derivation morphisms in the base category. Therefore, we can conclude the following: 

\begin{proposition}\label{prop:mon=der-arrow} We have an isomorphism of category $\mathsf{CMON}\left[\mathsf{Arr}[\mathbb{X}] \right] \cong \mathsf{DER}[\mathbb{X}]$. 
\end{proposition}

\section{Differential Modality on Arrow Category}\label{sec:diff-mod}

Thanks to Sec \ref{sec:der-monoid},  we know what commutative monoids are in the arrow category and we can use this to lift our differential modality to the arrow category. So for the remainder of this section, let $\mathbb{X}$ be a differential category with differential modality $(\mathsf{S}, \mu, \eta, \mathsf{m}, \mathsf{u}, \mathsf{d})$ and assume that $\mathbb{X}$ also have finite biproducts. We already have our monad $(\overline{\mathsf{S}}, \overline{\mu}, \overline{\eta})$ on our arrow category, so we now give its algebra modality structure. 

Starting with the multiplication $\overline{\mathsf{m}}$, define \[\overline{\mathsf{m}}_{\left( \phi: A_0 \to A_1 \right)}: \overline{\mathsf{S}}\left( \phi: A_0 \to A_1 \right) \boxtimes \overline{\mathsf{S}}\left( \phi: A_0 \to A_1 \right) \to \overline{\mathsf{S}}\left( \phi: A_0 \to A_1 \right)\] as the pair of following maps: 
\begin{equation}\begin{gathered}\label{diag:overline-m} 
\xymatrixcolsep{3pc}\xymatrix{ \mathsf{S}(A_0) \otimes \mathsf{S}(A_0) \ar[r]^-{\mathsf{m}_{A_0}} &\mathsf{S}(A_0) 
  } \\
  \xymatrixcolsep{9pc}\xymatrix{ {\begin{matrix} \left( \mathsf{S}(A_0) \otimes \mathsf{S}(A_0) \otimes A_1 \right) \\ \oplus \left( \mathsf{S}(A_0) \otimes A_1 \otimes \mathsf{S}(A_0) \right) \end{matrix}} \ar[rr]^-{\begin{bmatrix} \mathsf{m}_{A_0} \otimes 1_{A_1} & (\mathsf{m}_{A_0} \otimes 1_{A_1}) \circ (1_{\mathsf{S}(A_0)} \otimes \sigma_{A_1, \mathsf{S}(A_0)}) \end{bmatrix}} && \mathsf{S}(A_0) \otimes A_1.
  } 
\end{gathered}\end{equation}

For the unit $\overline{\mathsf{u}}$, define \[\overline{\mathsf{u}}_{\left( \phi: A_0 \to A_1 \right)}: \left( 0: I \to \mathsf{0} \right) \to \overline{\mathsf{S}}\left( \phi: A_0 \to A_1 \right)\] as follows: 
\begin{align}
\overline{\mathsf{u}}_{\left( \phi: A_0 \to A_1 \right)} := \left( \mathsf{u}_{A_0}: I \to \mathsf{S}(A_0), 0: \mathsf{0} \to A_1 \right).  \end{align}

\begin{proposition} $(\overline{\mathsf{S}}, \overline{\mu}, \overline{\eta}, \overline{\mathsf{m}}, \overline{\mathsf{u}})$ is an algebra modality. 
\end{proposition}
\begin{proof} We first explain why $\overline{\mathsf{S}}\left( \phi: A_0 \to A_1 \right)$ is a commutative monoid in the arrow category. To do so, consider that $\left( \overline{\mathsf{S}}\left( \phi: A_0 \to A_1 \right), \overline{\mu} \right)$ is an $\overline{\mathsf{S}}$-algebra. Applying Prop \ref{prop:overlineSalg-to-der} to this $\overline{\mathsf{S}}$-algebra, gives us that $(1 \otimes \phi) \circ \mathsf{d}: (\mathsf{S}(A_0), \mathsf{m}, \mathsf{u}) \to (\mathsf{S}(A_0) \otimes A_1, \mathsf{m} \otimes 1)$ is a $\mathsf{S}$-derivation, and so a derivation. Therefore by Prop \ref{prop:mon=der-arrow}, we get that:
\[\left( \left( (1 \otimes \phi) \circ \mathsf{d}: \mathsf{S}(A_0) \to \mathsf{S}(A_0) \otimes A_1 \right), \left(\mathsf{m}, \begin{bmatrix} \mathsf{m} \otimes 1 & (\mathsf{m} \otimes 1) \circ \sigma \end{bmatrix} \right) , (\mathsf{u}, 0) \right)\]
is a commutative monoid in the arrow category. Moreover, by symmetry coherences and commutativity of the multiplication, one can easily check that $(\mathsf{m} \otimes 1) \circ \sigma = (\mathsf{m} \otimes 1) \circ (1 \otimes \sigma)$. Therefore we may rewrite our above commutative monoid as the triple $\left( \overline{\mathsf{S}}\left( \phi: A_0 \to A_1 \right), \overline{\mathsf{m}}, \overline{\mathsf{u}} \right)$, as desired. Furthermore, this also gives us that $\overline{\mathsf{m}}$ and $\overline{\mathsf{u}}$ are indeed maps in the arrow category, so they are well-defined. Naturality of $\mathsf{m}$, $\sigma$, $\mathsf{u}$, and $0$ gives us that $\overline{\mathsf{m}}$ and $\overline{\mathsf{u}}$ are also natural transformations as desired. Lastly, it remains to explain why $\overline{\mu}$ is a monoid morphism. However since we know that $\overline{\mu}$ is an $\overline{\mathsf{S}}$-algebra morphism, by Prop \ref{prop:overlineSalg-to-der}, this means that $\overline{\mu}$ is also an $\mathsf{S}$-derivation morphism, which means that it is also a derivation morphism. Then by Prop \ref{prop:mon=der-arrow}, we get that $\overline{\mu}$ is a monoid morphism in the arrow category. So we conclude that $(\overline{\mathsf{S}}, \overline{\mu}, \overline{\eta}, \overline{\mathsf{m}}, \overline{\mathsf{u}})$ is an algebra modality, as desired. 
\end{proof}

Finally, it remains to give our deriving transformation $\overline{\mathsf{d}}$. So define \[\overline{\mathsf{d}}_{\left( \phi: A_0 \to A_1 \right)}: \overline{\mathsf{S}}\left( \phi: A_0 \to A_1 \right) \to \overline{\mathsf{S}}\left( \phi: A_0 \to A_1 \right) \boxtimes \left( \phi: A_0 \to A_1 \right)\] as the pair of following maps: 
\begin{equation}\begin{gathered}\label{diag:overline-d} 
\xymatrixcolsep{3pc}\xymatrix{ \mathsf{S}(A_0) \ar[r]^-{\mathsf{d}_{A_0}} &\mathsf{S}(A_0) \otimes A_0
  } \\
  \xymatrixcolsep{7pc}\xymatrix{ \mathsf{S}(A_0) \otimes A_1 \ar[rr]^-{\begin{bmatrix} 1_{\mathsf{S}(A_0)} \otimes 1_{A_1} \\ (1_{\mathsf{S}(A_0)} \otimes \sigma_{A_0, A_1}) \circ (\mathsf{d}_{\mathsf{S}(A_0)} \otimes 1_{A_1}) \end{bmatrix}} && (\mathsf{S}(A_0) \otimes A_1) \oplus (\mathsf{S}(A_0) \otimes A_1 \otimes A_0).
  }
\end{gathered}\end{equation}
\begin{theorem}\label{thm:arrow-diff-cat} $(\overline{\mathsf{S}}, \overline{\mu}, \overline{\eta}, \overline{\mathsf{m}}, \overline{\mathsf{u}}, \overline{\mathsf{d}})$ is a differential modality and therefore $\mathsf{Arr}[\mathbb{X}]$ is a differential category. 
\end{theorem}
\begin{proof} Let us first explain why $\overline{\mathsf{d}}$ is well-defined. Using the interchange rule \textbf{[D.5]}, we get that the following diagram commutes: 
\[ \xymatrixcolsep{5pc}\xymatrix{ \mathsf{S}(A_0) \ar[ddd]_-{\mathsf{d}} \ar[r]^-{\mathsf{d}} \ar@{}[ddr]|-{\text{\normalfont \textbf{[D.5]}}} & \mathsf{S}(A_0) \otimes A_0 \ar[r]^-{1 \otimes \phi} \ar[dd]|-{\begin{bmatrix} 1 \otimes 1 \\ (1 \otimes \sigma) \circ (\mathsf{d} \otimes 1) \end{bmatrix}}  & \mathsf{S}(A_0) \otimes A_1 \ar[ddd]|-{\begin{bmatrix} 1 \otimes 1  \\ (1 \otimes \sigma) \circ (\mathsf{d} \otimes 1) \end{bmatrix}} \\ 
& \\ 
& (\mathsf{S}(A_0) \otimes A_0) \oplus (\mathsf{S}(A_0) \otimes A_0 \otimes A_0) \ar[dr]^-{ ~~~~~~~~~\begin{bmatrix} 1 \otimes \phi & 0 \\ 0 & 1 \otimes \phi \otimes 1 \end{bmatrix} } &  \\ 
\mathsf{S}(A_0) \otimes A_0  \ar[ur]^-{ \begin{bmatrix} 1 \otimes 1 \\ \mathsf{d} \otimes 1 \end{bmatrix} } \ar[rr]_-{ \begin{bmatrix} 1 \otimes \phi \\ (1 \otimes \phi \otimes 1) \circ (\mathsf{d} \otimes 1) \end{bmatrix} } & & (\mathsf{S}(A_0) \otimes A_1) \oplus (\mathsf{S}(A_0) \otimes A_1 \otimes A_0).
  } \]
  Thus $\overline{\mathsf{d}}: \overline{\mathsf{S}}\left( \phi: A_0 \to A_1 \right) \to \overline{\mathsf{S}}\left( \phi: A_0 \to A_1 \right) \boxtimes \left( \phi: A_0 \to A_1 \right)$ is indeed a map in the arrow category. Clearly naturality of $\mathsf{d}$ implies that $\overline{\mathsf{d}}$ is a natural transformation as well. 

  \smallskip 
  
  So now we must show that $\overline{\mathsf{d}}$ satisfies the five deriving transformation axioms. However these are mostly straightforward by construction (though we will not write down all the diagrams explicitly as they are not necessarily more enlightening). Indeed from the constant rule \textbf{[D.1]} and the annihilation property of the zero map, we get that the constant rule \textbf{[D.1]} holds in the arrow category; using the Leibniz rule \textbf{[D.2]} and some obvious symmetry coherences, we get that the Leibniz rule also holds in the arrow category; using the linear rule \textbf{[D.3]} and the constant rule \textbf{[D.1]},  gives us that the linear rule \textbf{[D.3]} holds in the arrow category; while using the chain rule \textbf{[D.4]} and the Leinbiz rule \textbf{[D.2]} will give us that the chain rule \textbf{[D.4]} holds in the arrow category; and finally from the interchange rule \textbf{[D.5]} and again some obvious symmetry coherences, we get that the interchange rule also \textbf{[D.5]} holds in the arrow category. We leave it as an exercise for the motivated reader to write down all the necessary diagrams. So we conclude that $\overline{\mathsf{d}}$ is a deriving transformation and thus the arrow category is a differential category. 
\end{proof}

In summary, for a differential category with finite biproducts, its arrow category will again be a differential category (with finite biproducts). This gives us a novel source of examples of differential categories.

\smallskip

Now recall that we mentioned that differential categories were first introduced to give the categorical semantics of differential linear logic \cite{ehrhard2017introduction}. However, in such models, we require more from our differential modality; namely, that the canonical natural transformations (in the presence of biproducts) be isomorphisms. These are called the \textit{Seely isomorphisms}, which are a fundamental notion in the categorical semantics of (Differential) Linear Logic. So we conclude this section with a brief discussion about the Seely isomorphisms. 

\begin{definition}\label{def:Seely} For an additive symmetric monoidal category $\mathbb{X}$ with finite biproducts, a \textbf{differential storage modality} \cite[Dual of Def 10]{Blute2019} is a differential modality $(\mathsf{S}, \mu, \eta, \mathsf{m}, \mathsf{u}, \mathsf{d})$ such that the canonical natural transformations $\chi_{A,B}: \mathsf{S}(A) \otimes \mathsf{S}(B) \to \mathsf{S}(A \oplus B)$ and $\chi_0: I \to \mathsf{S}(0)$, called the \textbf{Seely maps}, defined respectively as follows: 
\begin{equation}\begin{gathered}\label{diag:seely} 
\xymatrixcolsep{5pc}\xymatrix{ \chi_{A,B}: \mathsf{S}(A) \otimes \mathsf{S}(B) \ar[rr]^-{\mathsf{S}\left( \begin{bmatrix} 1_A \\ 0 \end{bmatrix} \right) \otimes \mathsf{S}\left( \begin{bmatrix} 0 \\ 1_B \end{bmatrix} \right)} && \mathsf{S}(A \oplus B) \otimes  \mathsf{S}(A \oplus B) \ar[r]^-{\mathsf{m}_{A \oplus B}} & \mathsf{S}(A \oplus B) 
  } \\
  \xymatrixcolsep{3pc}\xymatrix{ \chi_0: I \ar[r]^-{\mathsf{u}_\mathsf{0}} & \mathsf{S}(\mathsf{0}) 
  }
\end{gathered}\end{equation}
are natural isomorphisms. In this case, we call $\chi$ and $\chi_0$ the \textbf{Seely isomorphisms}. 
\end{definition}

\begin{example} $\mathsf{Sym}$ is a differential storage modality \cite[Ex 1]{Blute2019}, since famously the symmetric algebra satisfies that $\mathsf{Sym}(V\otimes W) \cong \mathsf{Sym}(V) \otimes \mathsf{Sym}(W)$ and $\mathsf{Sym}(\mathsf{0}) \cong \mathbb{K}$ \cite[Chap 16, Prop 8.2]{lang2002algebra}.
\end{example}

\begin{example} $\mathsf{S}^\infty$ is not a differential storage modality \cite[Remark 5.16]{cruttwell2019integral}, since for example, it is well known that $\mathcal{C}^\infty(\mathbb{R}^2)$ is not isomorphic to $\mathcal{C}^\infty(\mathbb{R}) \otimes \mathcal{C}^\infty(\mathbb{R})$.
\end{example}

In the arrow category, it is straightforward to work out that the canonical Seely map \[\overline{\chi}_{\left( \phi: A_0 \to A_1 \right),\left( \psi: B_0 \to B_1 \right)}: \overline{\mathsf{S}}\left( \phi: A_0 \to A_1 \right) \boxtimes \overline{\mathsf{S}}\left( \psi: B_0 \to B_1 \right) \to \overline{\mathsf{S}}\left( (\phi: A_0 \to A_1) \oplus (\psi: B_0 \to B_1) \right)\] is given by the following pair of maps: 
\begin{equation}\begin{gathered}\label{diag:seely-arrow} 
\xymatrixcolsep{3pc}\xymatrix{ \mathsf{S}(A_0) \otimes \mathsf{S}(B_0) \ar[r]^-{\chi_{A_0,B_0}} & \mathsf{S}(A_0 \oplus B_0) 
  } \\
  \xymatrixcolsep{9.5pc}\xymatrix{ {\begin{matrix}(\mathsf{S}(A_0) \otimes \mathsf{S}(B_0) \otimes B_1) \\ \oplus (\mathsf{S}(A_0) \otimes A_1 \otimes \mathsf{S}(B_0)) \end{matrix}} \ar[rr]^-{\begin{bmatrix} \chi_{A_0,B_0} \otimes 1_{B_1} & (\chi_{A_0,B_0} \otimes 1_{A_1}) \circ (1_{\mathsf{S}(A_0)} \otimes \sigma_{A_1, \mathsf{S}(B_0)}) \end{bmatrix}} && \mathsf{S}(A_0 \oplus B_0) \otimes (A_1 \oplus B_1) 
  }
\end{gathered}\end{equation}
while $\overline{\chi}_0: \left(0: I \to \mathsf{0} \right) \to \overline{\mathsf{S}}(0: \mathsf{0} \to \mathsf{0})$ is $\overline{\mathsf{u}}_{(0: \mathsf{0} \to \mathsf{0})}$. Clearly if $\chi$ and $\chi_0$ are isomorphisms in $\mathbb{X}$, we get that $\overline{\chi}$ and $\overline{\chi}_0$ are also isomorphisms. So we conclude with the following observation: 

\begin{proposition} If $(\mathsf{S}, \mu, \eta, \mathsf{m}, \mathsf{u}, \mathsf{d})$ is a differential storage modality, then $(\overline{\mathsf{S}}, \overline{\mu}, \overline{\eta}, \overline{\mathsf{m}}, \overline{\mathsf{u}}, \overline{\mathsf{d}})$ is also a differential storage modality. 
\end{proposition}

\section{A Tangent Category and CDC of Derivations}\label{sec:tan}

In the previous section we showed that for a differential category with biproducts, its differential modality $\mathsf{S}$ lifts to a differential modality on the arrow category, and so the arrow category is also a differential category with finite biproducts. There are some important immediate consequences of this fact in that we get a \textit{tangent category} and a \textit{cartesian differential category} (CDC) of $\mathsf{S}$-derivations. We will not review tangent categories and cartesian differential categories in full here, instead we invite the reader to see \cite{cockett2014differential,cockett_et_al:LIPIcs:2020:11660} for an introduction to tangent categories and \cite{blute2009cartesian,cockett2014differential} for an introduction on cartesian differential categories. Also see \cite[Figure 1]{cockett_et_al:LIPIcs:2020:11660} for a practical map of the theory of differential categories, which highlights all the connections between differential categories, tangent categories, and cartesian differential categories. 

\smallskip

Now for the remainder of this section, again let $\mathbb{X}$ be a differential category with differential modality $(\mathsf{S}, \mu, \eta, \mathsf{m}, \mathsf{u}, \mathsf{d})$ and assume that $\mathbb{X}$ also have finite biproducts.

\smallskip

For a differential modality, its algebras form a \textit{tangent category}. Very briefly, a \textbf{tangent category} \cite{cockett2014differential} is a category equipped with an endofunctor $\mathsf{T}$, called the \textbf{tangent bundle functor} which should be thought of as assigning each object $A$ to its ``abstract tangent bundle" $\mathsf{T}(A)$. This tangent bundle functor comes equipped with various natural transformations which capture key properties of the tangent bundle from differential geometry. 

\smallskip

In our setting, the Eilenberg-Moore category $\mathsf{S}\text{-}\mathsf{ALG}$ of a differential modality $\mathsf{S}$ is a tangent category \cite[Thm 22]{cockett_et_al:LIPIcs:2020:11660}, where for an $\mathsf{S}$-algebra $(A, \nu)$, its tangent bundle is the $\mathsf{S}$-algebra $\mathsf{T}(A, \nu) = (A \oplus A, \nu^\flat)$, where the $\mathsf{S}$-algebra structure map $\nu^\flat: \mathsf{S}(A \oplus A) \to A \oplus A$ is defined as follows: 

\begin{align}
  \xymatrixcolsep{10pc}\xymatrix{ \mathsf{S}(A \oplus A) \ar[rr]^-{\begin{bmatrix} \nu \circ \mathsf{S}\left( \begin{bmatrix} 1_A & 0 \end{bmatrix} \right) \\
\mathsf{m}^\nu \circ (\nu \otimes 1_A) \circ \left( \mathsf{S}\left( \begin{bmatrix} 1_A & 0 \end{bmatrix} \right) \otimes \begin{bmatrix} 0 & 1_A \end{bmatrix} \right) \circ \mathsf{d}_{A \oplus A} \end{bmatrix}} && A \oplus A.
  }
\end{align}
As explained in \cite[Sec 5]{cockett_et_al:LIPIcs:2020:11660}, one should think of this as a sort of ``dual numbers construction" on an $\mathsf{S}$-algebra (as we will see in the examples below). In fact, the induced multiplication $\mathsf{m}^{\nu^\flat}$ for $\mathsf{T}(A, \nu)$ is precisely a generalized version of the dual numbers nilpotent multiplication. 
\smallskip

Now applying this construction to our differential modality on the arrow category gives us that $\overline{\mathsf{S}}\text{-}\mathsf{ALG}$ is also a tangent category. In other words, by passing through Thm \ref{thm:deri=alg} gives us that $\mathsf{S}$-derivations form a tangent category with tangent bundle functor worked out as follows: 

\begin{corollary}\label{cor:tan} $\mathsf{S}\text{-}\mathsf{DER}$ is a tangent category where for an $\mathsf{S}$-derivation $\mathsf{D}: (A, \nu) \to (M, \alpha)$, 
\[\mathsf{T}\left( \mathsf{D}: (A, \nu) \to (M, \alpha) \right) = \left( \begin{bmatrix} \mathsf{D} & 0 \\ 0 & \mathsf{D} \end{bmatrix}: (A \oplus A, \nu^\flat) \to (M \oplus M, \alpha^\flat)\right)\] 

where the action $\alpha^\flat: (A \oplus A) \otimes (M \oplus M) \to M \oplus M$ is induced from the following map (up to distributivity isomorphism): 
\begin{align}
  \xymatrixcolsep{5pc}\xymatrix{ (A \otimes M) \oplus (A \otimes M) \oplus (A \otimes M)  \oplus (A \otimes M)  \ar[rr]^-{\begin{bmatrix} \alpha & 0 & 0 & 0 \\ 
  0 & \alpha & \alpha & 0 \end{bmatrix}} && M \oplus M.
  }
\end{align}
\end{corollary}

\begin{example}\label{ex:tan-poly}\textbf{\textsf{A tangent category of ordinary derivations.}} $\mathsf{Sym}\text{-}\mathsf{ALG}$ is isomorphic to the category of commutative $\mathbb{K}$-algebras, which is a tangent category whose tangent bundle functor is given by dual numbers \cite[Ex 3.3]{cockett_et_al:LIPIcs:2020:11660}. Now the category of $\mathsf{Sym}$-derivations is of course simply the category of ordinary derivations, so $\mathsf{Sym}\text{-}\mathsf{DER} \simeq \mathsf{DER}\left[ \mathsf{VEC}_\mathbb{K} \right]$ (Ex \ref{ex:der-sym}). Applying the above construction gives us that derivations form a tangent category again given by dual numbers. Indeed, recall that for a commutative $\mathbb{K}$-algebra $A$, its ring of dual numbers is the commutative $\mathbb{K}$-algebra $A[\epsilon]= \lbrace a + b\epsilon \vert~ a,b \in A, \epsilon^2=0 \rbrace$. Now for an $A$-module $M$, define the $A[\epsilon]$-module $M[\epsilon] = \lbrace m + n \epsilon \vert~ m,n \in M \rbrace$, with action induced by $\epsilon^2=0$. Then the tangent bundle for derivations sends a derivation $\mathsf{D}:A \to M$ to the derivation $\mathsf{D}[\epsilon]: A[\epsilon] \to M[\epsilon]$ defined as $D[\epsilon](a + b \epsilon) = \mathsf{D}(a) + \mathsf{D}(b)\epsilon$. 
\end{example}

\begin{example}\label{ex:tan-smooth}\textbf{A tangent category of \textsf{$\mathcal{C}^\infty$-derivations.}} $\mathsf{S}^\infty\text{-}\mathsf{ALG}$ is isomorphic to the category of $\mathcal{C}^\infty$-rings, which is a tangent category whose tangent bundle functor is again given by dual numbers \cite[Ex 23.2]{cockett_et_al:LIPIcs:2020:11660}. Indeed, for a $\mathcal{C}^\infty$-ring $A$, its ring of dual numbers $A[\epsilon]$ is again a $\mathcal{C}^\infty$-ring where for every smooth function $f: \mathbb{R}^n \to \mathbb{R}$, $\Phi[\epsilon]_f: A[\epsilon]^n \to A[\epsilon]$ is defined as follows: 
\[ \Phi[\epsilon]_f(a_1 + b_1 \epsilon, \hdots, a_n + b_n \epsilon) = \Phi_f(a_1, \hdots, a_n) + \sum \limits_{i=1}^{n} \frac{\partial f}{\partial x_i}(a_1, \hdots,a_n)b_i \epsilon. \]
 Now $\mathsf{S}^\infty\text{-}\mathsf{DER}$ is isomorphic to the category of $\mathcal{C}^\infty$-derivations (Ex \ref{ex:der-smooth}), which will be a tangent category whose tangent bundle functor sends a $\mathcal{C}^\infty$-derivation $\mathsf{D}:A \to M$ to the $\mathcal{C}^\infty$-derivation $\mathsf{D}[\epsilon]: A[\epsilon] \to M[\epsilon]$, where $M[\epsilon]$ and $\mathsf{D}[\epsilon]$ are defined as in the previous example. 
\end{example}

On the other hand, the free algebras of a differential modality form a \textit{cartesian differential category}. Very briefly, a \textbf{cartesian differential category} \cite{blute2009cartesian} is a category with finite products $\times$ equipped with a \textbf{differential combinator} $\mathsf{D}$ which is an operator which for every map $f: A \to B$ produces its derivative $\mathsf{D}[f]: A \times A \to B$. It is also worth mentioning that a cartesian differential category is also a tangent category \cite[Prop 4.7]{cockett2014differential}. 

\smallskip

In our setting, the \textit{opposite} of the Kleisli category of our differential modality is a cartesian differential category \cite[Prop 3.2.1]{blute2009cartesian}. Recall that for a monad $(\mathsf{S}, \mu, \eta)$, its Kleisli category $\mathsf{KL}(\mathsf{S})$ is the category whose objects are the same as $\mathbb{X}$ but where a map from $A \to B$ in the Kleisli category is a map of type $A \to \mathsf{S}(B)$ in $\mathbb{X}$. Alternatively, as is well-known, the Kleisli category $\mathsf{KL}(\mathsf{S})$ is equivalent to the full subcategory of free $\mathsf{S}$-algebras $(\mathsf{S}(A), \mu_A)$, which we will denote by $\mathsf{S}\text{-}\mathsf{ALG}_{\text{free}}$, so $\mathsf{KL}(\mathsf{S}) \simeq \mathsf{S}\text{-}\mathsf{ALG}_{\text{free}}$. 

\smallskip

For our differential modality $\mathsf{S}$, the opposite category of $\mathsf{KL}(\mathsf{S})$ (equiv. $\mathsf{S}\text{-}\mathsf{ALG}_{\text{free}}$) is a cartesian differential category or, in other words, one could say that $\mathsf{KL}(\mathsf{S})$ (equiv. $\mathsf{S}\text{-}\mathsf{ALG}_{\text{free}}$) is a cocartesian differential category. So for a Kleisli map $f: A \to \mathsf{S}(B)$, the differential combinator would send it to a map $\mathsf{D}[f]: A \to \mathsf{S}(B \oplus B)$ (where here recall the product in the opposite category $\mathsf{KL}(\mathsf{S})$ would be the coproduct, which is given by the biproduct $\oplus$ of our base category but which is not a biproduct in the Kleisli category) which is defined as the following composite: 
\begin{align}
  \xymatrixcolsep{5pc}\xymatrix{ A \ar[r]^-{f} & \mathsf{S}(B) \ar[r]^-{\mathsf{d}_B} & \mathsf{S}(B) \otimes B \ar[r]^-{1_{\mathsf{S}(B)} \otimes \eta_B} & \mathsf{S}(B) \otimes \mathsf{S}(B) \ar[r]^-{\chi_{B,B}} & \mathsf{S}(B \oplus B)
  }
\end{align}
where $\chi$ is the Seely map as defined in (\ref{diag:seely}). 

\smallskip

In the case of $(\overline{\mathsf{S}}, \overline{\mu}, \overline{\eta})$, the free $\overline{\mathsf{S}}$-algebras correspond to $\mathsf{S}$-derivations of the form
\[(1_{\mathsf{S}(A_0)} \otimes \phi) \circ \mathsf{d}_{A_0}: (\mathsf{S}(A_0), \mu_{A_0}) \to (\mathsf{S}(A_0) \otimes A_1, \mathsf{m}_{A_0} \otimes 1_{A_1})\]
for some map $\phi: A_0 \to A_1$. Then let $\mathsf{S}\text{-}\mathsf{DER}_{\text{free}}$ be the full subcategory of $\mathsf{S}\text{-}\mathsf{DER}$ of this form. Thus we have that $\mathsf{KL}(\overline{\mathsf{S}}) \simeq \overline{\mathsf{S}},\text{-}\mathsf{ALG}_{\text{free}} \cong \mathsf{S}\text{-}\mathsf{DER}_{\text{free}}$, and so we may conclude that:

\begin{corollary}\label{cor:CDC} The opposite category of $\mathsf{S}\text{-}\mathsf{DER}_{\text{free}}$ is a cartesian differential category. 
\end{corollary}

\bibliographystyle{plain}     
\bibliography{references}   

@article{blute2006differential,
  title={{Differential Categories}},
  author={Blute, R. F. and Cockett, J. R. B. and Seely, R. A. G.},
  journal={{Mathematical Structures Computer Science}},
  year={2006},
  publisher={Cambridge Univ Press}
}

@article{blute2015derivations,
  title={{Derivations in Codifferential Categories}},
  author={Blute, R. F. and Lucyshyn-Wright, R. B. B. and O'Neill, K.},
 journal={{Cahiers de Topologie et G{\'e}om{\'e}trie Diff{\'e}rentielle Cat{\'e}goriques}},
  volume={57},
  pages={243--280},
  year={2016}
}

@Article{Blute2019,
author="Blute, R.F.
and Cockett, J.R.B.
and Lemay, J-S. P.
and Seely, R.A.G.",
title="{Differential Categories Revisited}",
journal="{Applied Categorical Structures}",
year="2020",
volume="28", 
pages="171-235",
doi="10.1007/s10485-019-09572-y"}

@article{lemay2019differential,
  title={{Differential Algebras in Codifferential Categories}},
  author={Lemay, J.-S. P.},
  journal={{Journal of Pure and Applied Algebra}},
  volume={223},
  number={10},
  pages={4191--4225},
  year={2019},
  publisher={Elsevier}
}

@article{ehrhard2017introduction,
  title={{An introduction to Differential Linear Logic: proof-nets, models and antiderivatives}},
  author={Ehrhard, T.},
  journal={{Mathematical Structures Computer Science}},
  year={2017},
  publisher={Cambridge University Press}
}

@article{lang2002algebra,
  title={Algebra, revised 3rd ed},
  author={Lang, S.},
  journal={{Graduate Texts in Mathematics}},
  volume={211},
  year={2002}
}

@book{moerdijk2013models,
  title={{Models for smooth infinitesimal analysis}},
  author={Moerdijk, I. and Reyes, G. E.},
  year={2013},
  publisher={Springer Science \& Business Media}
}

@article{cruttwell2019integral,
  title={{Integral and Differential structure on the Free {$C^\infty$}-ring Modality}},
  author={Cruttwell, G. S. H. and Lemay, J-S. P. and Lucyshyn-Wright, R. B. B.},
journal={{Cahiers de topologie et g{\'e}om{\'e}trie diff{\'e}rentielle cat{\'e}goriques}},
volume={62}, 
number={2}, 
pages={116--176},
  year={2021}
}

@phdthesis{o2017smoothness,
  title={{Smoothness in Codifferential Categories}},
  author={O'Neill, K.},
  year={2017},
  school={Universit{\'e} d'Ottawa/University of Ottawa}
}

@article{blute2011kahler,
  title={{K{\"a}hler Categories}},
  author={Blute, R. and Cockett, J. R. B. and Porter, T. and Seely, R. A. G.},
  journal={Cahiers de topologie et g{\'e}om{\'e}trie diff{\'e}rentielle cat{\'e}goriques},
  volume={52},
  number={4},
  pages={253--268},
  year={2011}
}

@article{dubuc19841,
  title={{On 1-form classifiers}},
  author={Dubuc, E.and Kock, A.},
  journal={{Communications in Algebra}},
  volume={12},
  number={12},
  pages={1471--1531},
  year={1984},
  publisher={Taylor \& Francis}
}

@article{joyce2011introduction,
  title={{An introduction to {$C^\infty$}-schemes and {$C^\infty$}-algebraic geometry}},
  author={Joyce, D.},
  journal={{Surveys in Differential Geometry}},
  volume={17},
  year={2012}
}

@article{blute2009cartesian,
  title={{Cartesian Differential Categories}},
  author={Blute, R. F. and Cockett, J. R. B. and Seely, R. A. G.},
  journal={{Theory and Applications of Categories}},
  volume={22},
  number={23},
  pages={622--672},
  year={2009}
}

@article{cockett2014differential,
  title={{Differential Structure, Tangent Structure, and {SDG}}.},
  author={Cockett, J. R. B. and Cruttwell, G. S. H.},
  journal={{Applied Categorical Structures}},
  volume={22},
  number={2},
  pages={331--417},
  year={2014}
}

@InProceedings{cockett_et_al:LIPIcs:2020:11660,
  author =	{Cockett, J. R. B. and Lemay, J.-S. P. and Lucyshyn-Wright, R. B. B.},
  title =	{{Tangent Categories from the Coalgebras of Differential Categories}},
  booktitle =	{{28th EACSL Annual Conference on Computer Science Logic (CSL 2020)}},
  pages =	{17:1--17:17},
  series =	{{Leibniz International Proceedings in Informatics (LIPIcs)}},
  ISBN =	{978-3-95977-132-0},
  ISSN =	{1868-8969},
  year =	{2020},
  volume =	{152},
  publisher =	{Schloss Dagstuhl--Leibniz-Zentrum fuer Informatik},
  address =	{Dagstuhl, Germany},
  doi =		{10.4230/LIPIcs.CSL.2020.17},
}

@article{riehl2013theory,
  title={{The Theory and Practice of Reedy Categories}},
  author={Riehl, E. and Verity, D.},
  journal={{Theory and Applications of Categories}},
  volume={29},
  number={9},
  pages={256--301},
  year={2014}
}

@book{riehl2014categorical,
  title={{Categorical Homotopy Theory}},
  author={Riehl, E.},
  volume={24},
  year={2014},
  publisher={Cambridge University Press}
}

\end{document}